\documentclass[a4paper,twoside,12pt,leqno]{article}
\usepackage{amsmath,amsthm,amssymb, enumerate}
\usepackage{epsfig, float, afterpage,array,multicol,eucal, amscd}
\swapnumbers
\theoremstyle{plain}
\newtheorem{Thm}{Theorem}[section]

\newtheorem{Lem}[Thm]{Lemma}
\newtheorem{Cor}[Thm]{Corollary}

\newtheorem*{Thm*}{Theorem}
\theoremstyle{remark}

\theoremstyle{definition}
\newtheorem{Rem}[Thm]{Remark}
\newtheorem{Ex}[Thm]{Example}
\newtheorem{Exs}[Thm]{Examples}
\newtheorem{Rems}[Thm]{Remarks}

\newtheorem{Not}[Thm]{Notation}
\newtheorem{Def}[Thm]{Definition}
\newtheorem{Defs}[Thm]{Definitions}

\newtheorem{Rev}[Thm]{Review}
\numberwithin{figure}{Thm}

\newcommand{\abs}[1]{\left\lvert#1\right\rvert} %absolute value
\newcommand{\norm}[1]{\left\lVert#1\right\rVert}
\newcommand{\gen}[1]{\langle\mkern3mu#1\mkern3mu\rangle}
\newcommand{\normgen}[1]{{\langle}{\mkern-3.7mu}{\lvert}{\mkern1.2mu}{#1}{\mkern1.2mu}{\rvert}{\mkern-3.7mu}{\rangle}}
\newcommand{\gp}[2]{\gen{{#1}\mid #2}}
\newcommand{\prs}[2]{\gen{#1\parallel #2}}
\newcommand{\lsup}[2]{{}^{#1}\mkern-1mu{#2}}

\renewcommand{\le}{\leqslant}
\renewcommand{\ge}{\geqslant}

\def\d1{\discretionary{-}{}{-}}
\def\ors{sur\-face\d1plus\d1one\d1re\-la\-tion }
\def\Ors{Sur\-face\d1plus\d1one\d1re\-la\-tion }
\def\ore{one\d1re\-la\-tor }
\def\Ore{One\d1re\-la\-tor }
\def\maxx{\text{max}}

\def\Z{\mathbb{Z}}
\def\Q{\mathbb{Q}}
\def\C{\mathbb{C}}
\def\R{\mathbb{R}}
\def\rfpap{{\operatorname{\scriptstyle{\text{\bf R}}}}{\mathcal{F}}\mkern -1.2mu{p}{\forall}{p}}
\def \leftmod {{\setminus}}
\def \rightmod {{/}}
\def\onto{\twoheadrightarrow}
\def\coloneqq{\mathrel{\mathop\mathchar"303A}\mkern-1.2mu=}
\def\eqqcolon{=\mkern-1.2mu\mathrel{\mathop\mathchar"303A}}

\DeclareMathOperator{\glue}{glue}
\DeclareMathOperator{\axis}{axis}
\DeclareMathOperator{\Eaxis}{\text{$E$}axis}
\DeclareMathOperator{\Hop}{H}
\DeclareMathOperator{\supp}{supp}
\DeclareMathOperator{\tr}{tr}
 \DeclareMathOperator{\Tor}{Tor}
\DeclareMathOperator{\cd}{cd}
\DeclareMathOperator{\vcd}{vcd}
\DeclareMathOperator{\FL}{FL}

\DeclareMathOperator{\uE}{\underline{E}}
\DeclareMathOperator{\rank}{rank}
\begin{document}

\pagestyle{myheadings}
\markboth{\Ors groups}{Y. Antol\'{\i}n, W. Dicks and P.\ A.\ Linnell}
\title{Non-orientable \ors groups}

\author{Yago Antol\'{\i}n, Warren Dicks and  Peter A.\ Linnell}

\date{\today}

\maketitle

 \hfill\textit{To the memory of Karl Gruenberg}

\begin{abstract}
Recently Dicks-Linnell determined
the $L^2$-Betti numbers of the orientable \ors groups, and
their arguments involved some results that were obtained topologically by Hempel and Howie.
Using algebraic arguments, we now extend all these results of Hempel and Howie
to a larger class of two-relator groups, and we then apply the extended results
 to determine the $L^2$-Betti numbers of the  non-ori\-ent\-able \ors groups.

\medskip

{\footnotesize

}
\end{abstract}

\maketitle

\section{Notation} \label{Sbackground}

In this section, we collect together the conventions and basic notation we shall use.

Let $G$ be a (discrete) multiplicative group, fixed throughout the article.

For two subsets $A$, $B$ of a set $X$, the complement of $A \cap B$ in $A$ will be denoted
by $A-B$ (and not by $A\setminus B$ since we let
 $G\backslash Y$ denote the set of $G$-orbits of a left $G$-set $Y$).

By an \textit{ordering},  $<$, of a set, we shall mean a  binary relation  which \textit{totally} orders
the set.

A  \textit{sequence} is a \textit{set} endowed with a specified listing of its elements,
usually represented as a vector in which the coordinates are the elements of
(the underlying set of) the sequence.
For two sequences $A$, $B$,  their  concatenation will be denoted~$A \vee B$.

We use $\mathbb{R} \cup \{-\infty,\infty\}$
with the usual conventions, such as
$\frac{1}{\infty} = 0$.

We will find it useful to have a notation  for intervals in $\Z$ that is different from the notation for intervals in $\R$.

Let $i$, $j \in \Z$.

We write $$[i{\uparrow}j]\coloneqq
 \begin{cases}
(i,i+1,\ldots, j-1, j) \in \Z^{j-i+1}   &\text{if $i \le j$,}\\
() \in \Z^0 &\text{if $i > j$.}
\end{cases}
$$
Also, $[i{\uparrow}\infty[\,\,\,\, \coloneqq (i,i+1,i+2,\ldots)$ and
$[i{\uparrow}\infty]\coloneqq  [i{\uparrow}\infty[ \,\,\,\vee\,\, \{\infty\}$.
We define $[j{\downarrow}i]$ to be the reverse of the sequence $[i{\uparrow}j]$,
that is, $(j,j-1,\ldots,i+1,i)$.

We shall use sequence notation to define families of
indexed symbols.
Let $v$ be a symbol.
For each $k \in \Z$, we let $v_k$ denote the ordered pair
$(v,k)$.  We let $$v_{[i{\uparrow}j]}\coloneqq  \begin{cases}
(v_i,v_{i+1},   \cdots, v_{j-1},   v_j)  &\text{if $i \le j$,}\\
()  &\text{if $i > j$.}
\end{cases}$$
Also, $v_{[i{\uparrow}\infty[\,} \coloneqq (v_i,v_{i+1},v_{i+2},\ldots)$.
We define $v_{[j{\downarrow}i]}$ to be the reverse of the sequence~$v_{[i{\uparrow}j]}$.

Now suppose that $v_{[i{\uparrow}j]}$ is a sequence \textit{in} the group $G$,
that is, there is specified a map of sets $v_{[i{\uparrow}j]} \to G$.  We treat the
elements of $v_{[i{\uparrow}j]}$ as elements of $G$, possibly with repetitions,
and we define
\begin{align*}\Pi v_{[i{\uparrow}j]}&\coloneqq  \begin{cases}
v_i   v_{i+1}   \cdots v_{j-1}   v_j \in G   &\text{if $i \le j$,}\\
1 \in G &\text{if $i > j$.}
\end{cases}\\ \Pi v_{[j{\downarrow}i]} &\coloneqq  \begin{cases}
v_j   v_{j-1}   \cdots v_{i+1}   v_i \in G   &\text{if $j \ge i$,}\\
1 \in G &\text{if $j< i $.}
\end{cases}\end{align*}

For elements $a$, $b$ of  $G$,  we write $\overline a \coloneqq a^{-1}$,   $\lsup{a}{b} \coloneqq ab\overline  a$,
and $[a,b]\coloneqq ab\overline  a \overline b$.

For any subsets $R$ and $X$ of $G$, we let $\gen{R}$ denote the subgroup of $G$ generated by $R$,
we write
$\lsup{X}\!{R}\coloneqq \{\lsup{x}{r} \mid r \in R, x \in X\}$,
 and $G/\normgen{R} \coloneqq G/\gen{\lsup{G}\!{R}}$.
If $R = \{r\}$, we write simply $\gen{r}$, $\lsup{X}{r}$ and $G/\normgen{r}$, respectively.

For each set $X$,   the  \textit{vague cardinal} of $X$, denoted $\abs{X}$,  is
the element of $ [0{\uparrow}\infty]$ defined as follows.
If $X$ is a finite set, then $\abs{X}$ is defined to be  the
  cardinal of $X$, an element of $[0{\uparrow}\infty[$\,.
If $X$ is an infinite set, then $\abs{X}$ is defined to be $\infty$.

The \textit{rank} of~$G$, denoted $\rank(G)$, is the smallest element
of the subset of $[0{\uparrow}\infty]$ which
consists of vague cardinals  of  generating sets of $G$.

Mappings of left modules  will usually be
written on the right of their arguments.

\section{Background and summary of results} \label{Ssummary}

In outline, the article has the following structure.
More detailed definitions can be found in the appropriate sections.

Let $S$ be a surface group, that is, the fundamental group of a  closed  surface.
Here, there exists some $k \in [0{\uparrow}\infty[$ and
a presentation $S = \gp{x_{[1{\uparrow}k]}}{w}$ where
either $k$ is even and
$w = \prod_{i \in [1{\uparrow}\frac{k}{2}]} [x_{2i-1},x_{2i}]$ (the orientable case),
or $k \ge 1$ and $w = \prod x^2_{[1{\uparrow}k]}\vspace{1mm}$ (the non-orientable case).
 Let $r \in  \gp{x_{[1{\uparrow}k]}}{\quad}$ and let
$G\coloneqq \gp{x_{[1{\uparrow}k]}}{w, r}\vspace{1mm}$; we say that
$G$ is a \textit{\ors group}. A simple example is $\gp{a,b,c}{
a^2b^2c^2, abc} \simeq \gp{a,b}{[a,b]}$.
Under the name  `one\d1re\-la\-tor
surface groups',  the orientable \ors groups were introduced and
studied by Hempel~\cite{Hempel90}, and further investigated by
Howie~\cite{Howie04}, with the aim of carrying some of the known
theory of one-relator groups over to these two-relator groups.

The purpose of this article is to study
$G$
algebraically, and generalize the work of Hempel and Howie to include
the non-orientable case.
If $k \le 2$, then   $G$ is virtually abelian of rank at most two,
and we  consider such groups to be well understood.  Thus we assume that $k\ge 3$,  and here
the closed surface is said to be \textit{hyperbolic}.
Recall the following.
\begin{Lem}\label{Lem1} A free generating sequence of $\gp{a,b,c}{\quad}$ is
given by
$(x,y,z)\coloneqq (\overline b \overline a \,\overline c \overline b, ab, cab),$ and here
$\,\,\lsup{abcab}{([x,y]z^{2})} = (abcab) (\overline b \overline a \,\overline c \overline b)
(ab)(bcab)(\overline b \overline a)(cab)(cab) (\overline b \overline a\, \overline c \overline b \overline a)=
a^2b^2c^2.$ \hfill\qed
\end{Lem}
\noindent Thus, on setting  $d = k-2 \in [1{\uparrow}\infty[$,
we can change the generating sequence  from $x_{[1{\uparrow}k]}$
to   $(x,y) \vee z_{[1{\uparrow}d]}$ and arrange for
$w$ to take  the form $[x,y]u$, with $u \in \gen{ z_{[1{\uparrow}d]}}$.

This leads us to consider
the class of two-relator groups described as follows.
Let\linebreak $d \in [0{\uparrow}\infty[$\,, let  $F\coloneqq \gp{(x,y) \vee z_{[1{\uparrow}d]}}{\quad}$,
let $u$ and $r$ be elements of $F$,  suppose that $u \in  \gen{ z_{[1{\uparrow}d]}}$,
let
$S\coloneqq \gp{(x,y) \vee z_{[1{\uparrow}d]}}{[x,y]u} \quad \text{and let} \quad
G  \coloneqq  \gp{(x,y) \vee z_{[1{\uparrow}d]}}{[x,y]u,\,\, r}.$
Notice that our surface group has been changed to a form
 which includes many new groups, while
 we   lose three closed
surfaces, namely the sphere, the projective plane, and the Klein bottle.

In Section~\ref{sec:basic}, we introduce a rewriting procedure for $S$.

In Section~\ref{sec:potent}, we recall the concepts of potency and `being residually
a finite $p$-group for each prime~$p$', and we show that
 $S$ enjoys these properties,
and we discuss connections with early work of Karl Gruenberg.
The potency of $S$
is used later in Section~\ref{sec:HNN}  to show that
$G$
is virtually torsion free.

In Section~\ref{sec:HR}, we shall see that by changing $(x,y)$
to a different free generating sequence of $\gp{x,y}{\quad}$  without changing $[x,y]$,
and then by carefully  changing $r$ without changing the conjugacy class of the
image of $r$ in  $S$, we
may assume that we have a presentation in which either
 $r = x^m$ for some $m \in [0{\uparrow}\infty[$\,,
or $d\ge 1$ and $r$ is what we shall call a `Hempel relator
for the presentation $\prs{(x,y) \vee z_{[1{\uparrow}d]}}{[x,y]u}$'.
Notice that we use a double bar to distinguish a presentation  from the group being presented.

In the case where $r=x^m$, we shall see that
$G$ is  virtually  one-relator, and we consider
such groups to be well understood.

The main part of the article then
examines the case where $r$ is a Hempel relator; here, we can
generalize  the results that Hempel and Howie obtained for hyperbolic  orientable
\ors groups.

In Section~\ref{sec:HNN}, for $r$  a Hempel relator,
we perform  what Howie calls `Hempel's trick' and express $G$ as an HNN-extension of
a one-relator group over an isomorphism between two free subgroups.
We deduce that if  $r$ generates a \textit{maximal} cyclic subgroup of $F$,   then
$G$
is locally indicable.
The proof  uses a deep result of Howie which he proved by topological methods;
we present a shorter proof, based on Bass-Serre theory, in an appendix.
Let $m$ denote  the index of  $\gen{r}$ in a maximal cyclic subgroup of~$F$;
we show that
\newline \centerline
{$G$  is (((free) $\rtimes$ (cyclic of order $m$)) by (locally indicable)).}

In Section~\ref{sec:exact}, for $r$ a Hempel relator,
we construct an exact sequence which gives
 a two-dimensional $\uE G$.

In Section~\ref{sec:VFL}, we recall the concept of VFL and show that
$G$ enjoys this property,
and we calculate   Euler characteristics.

In Section~\ref{sec:L2}, we apply all the foregoing Hempel-Howie-type results
to  calculate the $L^2$-Betti numbers of the   \ors groups;
for the orientable case this was done in~\cite[Theorem~5.1]{DicksLinnell}, in essentially the same way,
using the original Hempel-Howie results.

In the appendix, as mentioned, we use Bass-Serre theory to simplify proofs of some important results of
Howie on local indicability.

\section{Rewriting in  $S = \gp{(x,y) \vee z_{[1{\uparrow}d]}}{[x,y]u}$}\label{sec:basic}

We shall use the following at various points in the article.

\begin{Not}\label{Not:gen} Let $d \in [0{\uparrow}\infty[$\,, let
$x$, $y$ and $z$ be symbols, let
$F \coloneqq \gp{(x,y) \vee z_{[1{\uparrow}d]}}{\quad}$,
and  let $u$ and $r$ be elements of $F$.  Suppose that
$u \in \gen{z_{[1{\uparrow}d]}}$.
Let  $S \coloneqq  \gp{(x,y) \vee z_{[1{\uparrow}d]}}{[x,y]u}$, and let
$G \coloneqq \gp{(x,y) \vee z_{[1{\uparrow}d]}}{[x,y]u,\, r}$.
We shall denote the natural map $F \to S$  by\linebreak $w \mapsto w\text{\,mod}[x,y]u$.

 The two-relator group $G$ is
a one-relator quotient   of $S$, with relator  $r \text{\,mod}[x,y]u$.

Let
$N(F)\coloneqq \gp{\Z \times ((x) \vee z_{[1{\uparrow}d]}) }{\quad}$.
For each $i \in \Z$, we shall denote the natural map\vspace{-1mm}
$$\gp{(x) \vee z_{[1{\uparrow}d]} }{\quad} \quad\simeq \quad \gp{ \{i\} \times ((x) \vee z_{[1{\uparrow}d]})
 }{\quad} \quad \le \quad \gp{ \Z \times ((x) \vee z_{[1{\uparrow}d]}) }{\quad}$$
by $w \mapsto \lsup{i}{w}$.
Let $y$ denote the automorphism of $N(F)$ determined by the shifting bijection\vspace{-1mm}
$$(\,\lsup{i}{x})\,\,\, \vee \,\,\,\lsup{i}{z}_{[1{\uparrow}d]}  \quad \to \quad (\,\lsup{i+1}{x})
 \,\,\, \vee \,\,\,\lsup{i+1}{z}_{[1{\uparrow}d]} $$
 for all
$i \in \Z$.   Hence $C_\infty\coloneqq \gp{y}{\quad}$ acts on $N(F)$.
If we form the semidirect product
\linebreak $N(F) \rtimes C_\infty$, we get the presentation $\gp{(x,y) \vee z_{[1{\uparrow}d]}}{\quad }$.  Thus we may
 identify   $N(F) \rtimes C_\infty$ with $F$, and  $N(F)$ then becomes the normal subgroup of $F$
generated by $(x) \vee z_{[1{\uparrow}d]}$, that is,
the set of elements whose $y$-exponent sum, with respect to $(x,y) \vee z_{[1{\uparrow}d]}\vspace{1mm}$, is zero.
 Notice that, for each
$(i,w) \in \Z \times  \gp{(x) \,\,\,\vee\,\,\,  z_{[1{\uparrow}d]}}{\quad}\vspace{1mm}$,
we have identified $\lsup{i}{w}  = \lsup{y^i}\!{w}$.
Bearing in mind that $\Z$ is an abbreviation for $y^{\Z} = C_\infty$,  we shall  write
$(\,\lsup{\Z}{x})$ for $\Z \times (x)  $, and $\lsup{\Z}{z}_{[1{\uparrow}d]}$  for
$\Z \times z_{[1{\uparrow}d]}$.

Let $N(S) \coloneqq \gp{(\,\lsup{\Z}{x})\,\,\, \vee\,\,\, \lsup{\Z}{z}_{[1{\uparrow}d]}}
{ (\,\lsup{i+1}{\overline x} \cdot \lsup{i}{u} \cdot \lsup{i}{x}  \mid i \in \Z)}.$
The shifting action of $C_\infty = \gp{y}{\quad}$ on $N(F)$ induces a $C_\infty$-action
on $N(S)$, and we find that we can identify $N(S) \rtimes C_\infty$ with~$S$.
Thus $N(S)$ is the image of $N(F)\vspace{1mm}$ in $S$.

For each $j\in \Z$, let $y$ act on
$\gp{ (\lsup{j}{x} ) \,\,\,\vee\,\,\,    \lsup{\Z}{z}_{[1{\uparrow}d]}}{\quad}$
by $\lsup{j}{x} \mapsto \lsup{j}{u} \cdot \lsup{j}{x}$, and,
for each $\lsup{i}{z}_\ast  \in  \lsup{\Z}{z}_{[1{\uparrow}d]}$,
$\lsup{i}{z}_\ast  \mapsto  \lsup{i+1}{z}_\ast$.
We find that we can make the identification
$$S = \gp{ (\,\lsup{j}{x}) \,\,\,\vee\,\,\,    \lsup{\Z}{z}_{[1{\uparrow}d]}}{\quad}\rtimes C_\infty.$$
Thus $\gp{ (\,\lsup{j}{x}) \,\,\,\vee\,\,\,    \lsup{\Z}{z}_{[1{\uparrow}d]}}{\quad}$
 can be viewed as a free factor of $N(F) \le F$ which
maps bijectively to $N(S)$ in $S$.   By varying $j$, we get a family of embeddings of $N(S)$ in $N(F) \le F$.
\hfill\qed
\end{Not}

\section{Potency of $S = \gp{ (x,y) \vee z_{[1{\uparrow}d]}}{[x,y]u}$}\label{sec:potent}

In this section, we prove a useful fact about $S$ and recall some related history.

\begin{Def}  A group $S$ is said to be \textit{potent} if, for each $s \in S -\{1\}$ and each
$m \in [2{\uparrow}\infty[$\,, there exists a homomorphism from $S$ to some finite group which sends
$s$ to some element of order exactly $m$; in this event, we also say that $S$ is a potent group.
\hfill\qed
\end{Def}

The following fact, whose  proof  invokes two results of  Allenby~\cite{Allenby}, will be
applied in Section~\ref{sec:HNN} to show that  certain two-relator groups are virtually torsion free.

\begin{Lem}\label{Lem:Allenby}  For $u \in \gp{z_{[1{\uparrow}d]}}{\quad}$,
the group  $S = \gp{(x,y)\vee z_{[1{\uparrow}d]}}{[x,y]u}$  is potent.
\end{Lem}

\begin{proof}  In the case where
$u \ne 1$, $S = \gp{x,y}{\quad}\mathop{\ast}\limits_{[y,x] =   u} \gp{z_{[1{\uparrow}d]}}{\quad},$
and then $S$ is potent because
the free product of two free groups
amalgamating a non-trivial cyclic group  is potent~\cite[Section~4]{Allenby}.
(We can also use the latter result as a perverse reference for the fact that free groups are potent.)
Now we may assume that $u=1$, and, hence,
 $S$ is the free product of a rank-two,
 free-abelian group and a free group.
Observe that free-abelian groups are potent.
Recall that the free product of two
potent groups is potent~\cite[Theorem~2.4]{Allenby}.
(We can also use the latter result as a less perverse reference for the fact
 that free groups are potent.)
The result now follows.
\end{proof}

We dedicate the remainder of this section to setting the foregoing
results of Allenby into an historical context,
with particular emphasis on connections with the early research of Karl Gruenberg.
This digression will allow us
to explain how some of the techniques involved can be used to prove
the potency of $S$  directly.

\begin{Def}
Let us say that a group $S$ is $\rfpap$ if, for each prime $p$, $S$ is residually
a finite $p$-group, that is, $S$ embeds in a direct product of finite $p$-groups; in this event,
we also say that $S$ is an $\rfpap$-group.
\hfill\qed
\end{Def}

The history begins with free groups.

In 1935, Magnus obtained the following important result.
\newline\null\hskip 2cm (H1)  Every free group is residually torsion-free nilpotent.

Let us recall the method that Magnus used.

Let $R$ be an associative ring, let $X$ be a set
and
let $R\langle\langle X \rangle\rangle$ denote the ring of formal
power-series in the set $X$ of non-commuting  variables.  Each $f
\in R\langle\langle X \rangle\rangle$ has a unique expression as a
formal sum $\sum f_{[0{\uparrow}\infty[}$ where, for each
 $m \in [0{\uparrow}\infty[$\,, $ f_m$ is
homogeneous of degree $m$.  For each $n \in [0{\uparrow}\infty[$,
let $I_n$ denote  $$I_n(R,X) \coloneqq  \{f \in R\langle\langle X \rangle\rangle
\mid   f_m = 0 \text{ for all }m \in [0{\uparrow}n-1]\},$$ a closed ideal in
 $R\langle\langle X \rangle\rangle $.
For each $n \in [1{\uparrow}\infty[$,
let $U_n$ denote  $U_n(R,X) \coloneqq 1 + I_n$, a subgroup of the
group of units of $R\langle\langle X \rangle\rangle $.
If $m \in [1{\uparrow}n]$, then the natural map from $U_m$ to the group of units of
$R\langle\langle X \rangle\rangle/I_{n+1}$
has kernel $U_{n+1}$ and
image which can be denoted $1 + (I_m/I_{n+1})$. It follows that
$U_{n+1}$ is a normal subgroup of  $U_{1}$ and that
 $U_n/U_{n+1}$ is  a free $R$-module
which is central in $U_1/U_{n+1}$.

Consider now the case where $R = \Z$.  Here, $U_n/U_{n+1}$  is a free-abelian group,
and, hence,  $U_1(\Z,X)$ is
residually torsion-free nilpotent; see~\cite[IVa]{Magnus35}.
Also, by considering leading coefficients in $I_1$ for reduced expressions,
Magnus showed that
  $1+X$ freely generates a free subgroup of $U_1(\Z,X)$; see~\cite[I]{Magnus35}.
This completes the outline of Magnus' proof of (H1).

Gruenberg~\cite[p.~29]{Gruenberg} remarks that A.\ Mal'cev in 1949, M.\ Hall in 1950, and Takehasi in 1951,
independently found the following result.
\newline\null\hskip 2cm (H2) Every free group is   $\rfpap$.
\newline Let us recall how (H2) is related to power series.
Let $p$ be a prime number and let $R = \Z_p$,
the integers modulo
$p$. In the case where $X$ is finite,  $U_1(\Z_p,X)$ is residually a
finite $p$-group, since the $U_1/U_{n+1}$ are finite $p$-groups. In
the case where $X$ is infinite, we can retract $R\langle\langle X
\rangle\rangle$ onto $R\langle\langle Y \rangle\rangle$ for any
subset $Y$ of $X$, and it follows that  $U_1(\Z_p,X)$ is again
residually a finite $p$-group. Now, as Luis Paris pointed out to us,
a leading-coefficient argument again shows that
 $1+X$ freely generates a free subgroup of $U_1(\Z_p,X)$, and this proves (H2).

Since every non-trivial finite $p$-group has a central, hence normal, subgroup of order $p$,
it is not difficult to see the following result.
\newline\null\hskip 2cm (H3)  Every $\rfpap$-group  is potent.
\newline The result (H3) was presented implicitly by
Fischer\d1Karrass\d1Solitar in 1972; see~\cite[Proof of Theorem~2]{FKS}.
It was presented explicitly by Kim\d1McCarron in 1993; see~\cite[Lemma~2.2]{KimMcCarron}.
The facts (H2) and (H3) together prove the following result.
\newline\null\hskip 2cm (H4) Every free group is potent.
\newline The result (H4) was presented explicitly by Stebe in 1971 with a proof,  attributed to Passman, based on (H1);
see~\cite[Lemma 1]{Stebe}.  Independently, it was presented implicitly by
Fischer\d1Karrass\d1Solitar in 1972, with a proof based on (H2) and (H3); see~\cite[Proof of Theorem~2]{FKS}.

\begin{Rem}\label{Rem:Allenby} Let $d\in[1{\uparrow}\infty[$\,,\, let $u \in \gp{ z_{[1{\uparrow}d]}}{\quad}$, and let
$S\coloneqq \gp{ (x,y) \vee z_{[1{\uparrow}d]}}{[x,y]u}$.

We sketch an argument that shows how
 power series  can be used
to prove that $S$
is $\rfpap$.

Let $p$ be a prime number.

As in Notation~\ref{Not:gen}, $S =  N(S)  \rtimes C_\infty$ where $C_\infty = \gp{y}{\quad}$
and $N(S) = \gp{(x)\,\, \vee\,\,\, \lsup{\Z}{z}_{[1{\uparrow}d]} }{\,\,}$.  Here $y$ acts on $N(S)$ by
$x \mapsto \lsup{0}{u} \cdot x$,\,\, $\lsup{i}{z}_\ast \mapsto \lsup{i+1}{z}_\ast$.
Let $b$ be a symbol and choose an identification of the countably infinite set
$(x) \vee \lsup{\Z}{z}_{[1{\uparrow}d]} $  with the subset $1 + b_{[0{\uparrow}\infty[}$
of $\Z_p\langle\langle b_{[0{\uparrow}\infty[}\rangle\rangle\vspace{1mm}$.  Hence, by the $p$-analogue of
Magnus' result, we get an embedding of
$N(S)$ in the group of units of $\Z_p\langle\langle b_{[0{\uparrow}\infty[}\rangle\rangle\vspace{1mm}$.  The action of
$y$ on $N(S)$ extends uniquely to a
continuous automorphism of $\Z_p\langle\langle b_{[0{\uparrow}\infty[}\rangle\rangle\vspace{1mm}$, and we can form the
skew-group-ring, or skew Laurent-polynomial ring,
$(\Z_p\langle\langle b_{[0{\uparrow}\infty[}\rangle\rangle)[C_\infty]\vspace{1mm}$.  It is then
clear that $S$ embeds in the group of units of
$(\Z_p\langle\langle b_{[0{\uparrow}\infty[}\rangle\rangle)[C_\infty]\vspace{1mm}$.

Consider any $n \in [1{\uparrow}\infty[$.

Let $q\coloneqq p^n$ and let $C_{q^2}\coloneqq \langle y \mid y^{q^2} \rangle$.

We have the closed, $y$-invariant ideal
$I_{n+1} = I_{n+1}(\Z_p, b_{[0{\uparrow}\infty[})$ of
 $\Z_p\langle\langle b_{[0{\uparrow}\infty[}\rangle\rangle\vspace{1mm}$, and we can form
the skew-group-ring $(\Z_p\langle\langle
b_{[0{\uparrow}\infty[}\rangle\rangle/I_{n+1})[C_\infty]\vspace{1mm}$.
It is not difficult to check that the image of $(N(S))^q$ is
trivial.

Let $J_{n+1}$ denote the (closed, $y$-invariant) ideal of
 $\Z_p\langle\langle b_{[0{\uparrow}\infty[}\rangle\rangle$
generated by $$I_{n+1} \cup \{ \lsup{i}{z}_\ast - \lsup{i+q}{z}_\ast
\mid \lsup{i}{z}_\ast \in \lsup{\Z}{z}_{[1{\uparrow}d]}\}.$$ Again,
we can form the skew-group-ring
$(\Z_p\langle\langle
b_{[0{\uparrow}\infty[}\rangle\rangle/J_{n+1})[C_\infty]\vspace{1mm}$.
It is not difficult to check  that the $y^{q}$-action fixes the
image $\lsup{\Z_q}{z}_{[1{\uparrow}d]}$ of
$\lsup{\Z}{z}_{[1{\uparrow}d]}$, and that the $y^{q^{2}}$-action
fixes the image of~$x$. Thus we can form the skew-group-ring
$(\Z_p\langle\langle
b_{[0{\uparrow}\infty[}\rangle\rangle/J_{n+1})[C_{q^2}]\vspace{1mm}$
and find that the image of $S$ is a finite $p$-group. By varying
$n$, we see that $S$ is residually a finite $p$-group. By
varying~$p$, we see that $S$ is $\rfpap$. By (H3), $S$ is potent.
\hfill\qed
\end{Rem}

In 1957, Gruenberg proved the following two results; see~\cite[Theorem~2.1(i) and
Part (iii) of the Corollary  on p.44]{Gruenberg}.
\newline\null\hskip 2cm (H5) Every residually torsion-free nilpotent group is $\rfpap$.
\newline\null\hskip 2cm (H6) The free product of any family of $\rfpap$-groups is  $\rfpap$.
\newline Each of these sheds important light on (H2).

In 1981, Allenby obtained the following result; see~\cite[Theorem~2.4]{Allenby}.
\newline\null\hskip 2cm   (H7) The free product of any family of potent groups is potent.
\newline  This sheds light on (H4).

\bigskip

We next consider a  class of groups which contains $S$ when $u\ne 1$.

Let $A$ and $B$ be free groups, let $a \in A -\{1\}$,
and let $b \in B-\{1\}$.  The amalgamated free product $A \mathop{\ast}\limits_{a=b} B$
is called a \textit{cyclically-pinched one-relator group}.

In 1968, G.~Baumslag used power series to obtain  the following result; see~\cite[Theorem~1]{Baumslag68}.
\newline\null\hskip 2cm (H8) If  $a$ is not a proper power, and $B$ is cyclic,
then  $A \mathop{\ast}\limits_{a=b} B\vspace{-1mm}$
is residually\linebreak \null\hskip 3cm torsion-free nilpotent.
\newline In particular, by (H5), $A \mathop{\ast}\limits_{a=b} B$  is then $\rfpap$.
In 1998, Kim and Tang used this to
obtain the following result; see~\cite[Corollary~3.6]{KimTang}.
\newline\null\hskip 2cm  (H9) $A \mathop{\ast}\limits_{a=b} B$ is $\rfpap$ if and only if $a$ or
 $b$ is not a proper power.
\newline This gives the earliest proof we know of that $S$ is $\rfpap$, since
the case where $u \ne 1$ follows from (H9) with $A = \gp{x,y}{\quad}$, $B = \gp{z_{[1{\uparrow}d]}}{\quad}$,
$a = [y,x]$, $b=u$, since $[y,x]$ is not a proper power,
while the case where $u=1$ follows from (H6) and the fact that free-abelian groups are $\rfpap$.

 In 1981, Allenby obtained the following result; see~\cite[Section~4]{Allenby}.
\newline\null\hskip 1.8cm  (H10) $A \mathop{\ast}\limits_{a=b} B$ is potent.
\newline We have seen that (H10) and (H7) imply that $S$ is potent.

\bigskip

Finally, let us mention fundamental groups of closed surfaces.

In 1968, Chandler~\cite{Chandler} obtained the following result.
\newline\null\hskip 1.8cm  (H11) For every closed surface
except the projective plane and the Klein bottle,
\linebreak\null\hskip 3cm the
fundamental group is residually torsion-free nilpotent and, hence,
\linebreak\null\hskip 3cm by (H5), is $\rfpap$.

The following is a consequence of (H10).
\newline\null\hskip 1.8cm  (H12) For every closed surface except the projective plane, the
fundamental group \linebreak\null\hskip 3cm  is potent.

\section{Hempel relators}\label{sec:HR}

In this section, we  introduce Hempel relators and show how to find them.

\begin{Defs} Let $F$ be a group.  Let $r \in F$.
If $r \in F -\{1\}$ and $\gen{r}$ is contained in a unique maximal infinite, cyclic subgroup $C$ of $F$,
then there exist a unique $m\in [1{\uparrow}\infty[\,,$ and a unique $s\in F,$ such that
$m =  [C{:}\gen{g}]$,  $\gen{s} = C$ and $s^m = r$;
we then define  $\sqrt[F]{r} \coloneqq s$ and $\log_Fr \coloneqq m$,
and say that $s$ is \textit{the root} of $r$ in $F$.
 For $r = 1$, we define $\sqrt[F] 1 \coloneqq  1$ and $\log_F1 = \infty$.
In all other cases we say that $r$ \textit{does not have a root} in $F$.
We say that  $F$ \textit{has roots} if every element of $F$ has a root in $F$.
\hfill\qed
\end{Defs}

\begin{Lem}\label{Lem:root} With {\normalfont Notation~\ref{Not:gen}}, there is a
free generating sequence $(x',y')$ of $\gp{x,y}{\quad}$ such that $[x',y'] = [x,y]$,
and the $y'$-exponent sum of  $r$ with respect to $(x',y')\vee z_{[1{\uparrow}d]}$ is zero,
and the $x'$-exponent sum of   $r$ with respect to $(x',y')\vee z_{[1{\uparrow}d]}$ is nonnegative.

Moreover,  $S$ has roots and any element of the free subgroup $N(S)$
has the same root in both $S$ and   $N(S)$.
\end{Lem}

\begin{proof}
Let $a$ and $b$ respectively
denote the $x$- and $y$-exponent sums of $r$ with respect to $(x,y)\vee z_{[1{\uparrow}d]}$.
We want $a\ge0$ and $b = 0$.

Replacing   $x$ with $xy^{\pm1}$ fixes  $[x,y]$   and changes
$(a,b)$ to $(a,b \pm a)$.
Replacing   $y$ with $yx^{\pm1}$
fixes $[x,y]$  and changes
 $(a,b)$ to   $(a \pm b,b)$.
If $ab \ne 0$, then one or more of these four operations reduces  $\abs{a}+\abs{b}$.
By repeating such operations, we can arrange that $ab=0$.
If $a=0$, then we can alter $(0,b)$ to $(b,b)$ and then to $(b,0)$; thus we may assume that $b=0$.
Finally, if $a < 0$, we can successively
alter $(a,0)$ to $(a,a)$, $(0,a)$, $(-a,a)$, and $(-a,0)$.  Thus we may also assume that
$a \ge 0$.

We next show that $r \text{\,mod}[x,y]u$ has a root in $S$.
We may assume that $r \text{\,mod}[x,y]u$ lies in $ S - \{1\}$.
Let $(x',y')$ be a free generating sequence of $\gp{x,y}{\quad}$
such that $[x',y'] = [x,y]$,
and the $y'$-exponent sum of $r$ with respect to $(x',y')\vee z_{[1{\uparrow}d]}$ is zero.
To simplify notation, we  forget the original $(x,y)$, use $(x',y')$ as the new
generating sequence, and name it  $(x,y)$.
Thus we may assume that $r\text{\,mod}[x,y]u$
lies in $N(S)$, a free subgroup of $S$.  Hence $r\text{\,mod}[x,y]u$  has a root in
$N(S)$.  Any cyclic subgroup of $S$ containing $r\text{\,mod}[x,y]u$
maps to a finite, hence trivial, subgroup in $S/N(S) = C_\infty = \gp{y}{\quad}$.
 Hence, any cyclic subgroup of $S$ containing $r\text{\,mod}[x,y]u$
lies in $N(S)$.  Thus the root of $r\text{\,mod}[x,y]u$ in the free group $N(S)$ is a root in~$S$.
\end{proof}

\begin{Def} With {\normalfont Notation~\ref{Not:gen}}, let
$X_1 \coloneqq (\,\lsup{1}{x})\,\, \vee \,\,\lsup{[0{\uparrow}\infty[\,}{z}_{[1{\uparrow}d]}$.
We say that $r$ is a \textit{Hempel relator} for the presentation
  $\prs{(x,y)\vee z_{[1{\uparrow}d]}}{[x,y]u}$
if the following hold.
\begin{enumerate}[(R1)]
\vskip-0.6cm \null
\item $r  \in \gp{ X_1   }{\quad} \le N(F) \le F$,
\vskip-0.6cm \null
\item In
$ \gp{ X_1  }{\quad}$, $r$
is not conjugate to any element of $\gen{\,\lsup{0}{\overline u} \cdot \lsup{1}{x}}$ .
\vskip-0.6cm \null
\item With respect to  $X_1 \vspace{1mm} $,   $r$~is cyclically reduced.
\vskip-0.6cm \null
\item With respect to  $X_1 \vspace{1mm} $,
$r$ \textit{involves} some element of $\lsup{0}{z}_{[1{\uparrow}d]}$, that is, $r$ does not lie in the free factor
$\gp{ (\,\lsup{1}{x})\,\,  \vee\,\, \lsup{[1{\uparrow}\infty[\,}{z}_{[1{\uparrow}d]} }{\quad}$.\hfill\qed
\end{enumerate}
\end{Def}

We now want to show that for our  purposes we can assume that $r$ is a Hempel relator.

\begin{Lem}\label{Lem:Hempelments} With {\normalfont Notation~\ref{Not:gen}},
there exists an element $w$ of $F$, an element $v $ of $\gen{\lsup{F}{([x,y]u)}}$, and
 an automorphism $\alpha$ of $\gp{x,y}{\quad}$ which fixes $[x,y]$,
such that  $r'\coloneqq (v)(\lsup{w\alpha}{r})$
is either a non-negative power of $x$ or
a  Hempel relator for $\prs{(x,y)\vee z_{[1{\uparrow}d]}}{[x,y]u}$.

Here,  $G \simeq \gp{(x,y)\vee z_{[1{\uparrow}d]}}{[x,y]u,\, r'}$ and $\log_Fr' = \log_S (r \text{\normalfont \,mod}[x,y]u)$.
\end{Lem}

\begin{proof} We shall successively alter $r$ and,
to avoid extra notation, we shall use the same symbol $r$ to denote the altered element of $F$ at each stage.

 We first alter $r$ by  automorphisms of
$\gp{x,y}{\quad}$ which fix   $[x,y]$.
By Lemma~\ref{Lem:root},
we may assume that  the $y$-exponent sum of $r$ with respect to $(x,y)\vee z_{[1{\uparrow}d]}$ is zero,
and the $x$-exponent sum of   $r$ with respect to $(x,y)\vee z_{[1{\uparrow}d]}$ is nonnegative.
Hence $r \text{\normalfont \,mod}[x,y]u$ lies in the free subgroup $N(S)$.

For each $j \in \Z$, we can view $(\,\lsup{j}{x}) \vee \lsup{\Z}{z}_{[1{\uparrow}d]}$
as a free generating sequence of $N(S)$ and
express $r \text{\normalfont \,mod}[x,y]u$ as a word therein and lift the word back to a new
element of $N(F)$; this multiplies $r$ by an element of $\gen{\lsup{F}{([x,y]u)}}$.  We then
reduce the expression cyclically, which corresponds to conjugating $r$ in $F$.

If $r= \lsup{j}{x}^m$ for some $m \in \Z$, then $m \ge 0$.
By replacing $r$ with $\lsup{y^{-\!j}}{r}$, we may assume that $r = x^m$,
and the desired conclusions hold.

Thus, we may assume that, for each $j \in \Z$, the
cyclically reduced expression for $r$ in
$(\,\lsup{j}{x}) \,\,\vee\,\, \lsup{\Z}{z}_{[1{\uparrow}d]}$ involves some element of
 $\lsup{\Z}{z}_{[1{\uparrow}d]}$, and, hence,
 there is a unique smallest non-empty interval
 $[\mu_j{\uparrow}\nu_j]$ in $\Z$ such that the  cyclically reduced expression for $r$
in
$(\,\lsup{j}{x}) \vee \lsup{\Z}{z}_{[1{\uparrow}d]}$
is a word in $(\,\lsup{j}{x}) \,\,\vee\,\, \lsup{[\mu_j{\uparrow}\nu_j]}{z}_{[1{\uparrow}d]}$.

We claim that for some $k \in \Z$, $\mu_{k+1} = k$.

To pass from $(\,\lsup{j+1}{x})\,\, \vee\,\,
\lsup{\Z}{z}_{[1{\uparrow}d]} $ to $(\,\lsup{j}{x})\,\, \vee\,\,
\lsup{\Z}{z}_{[1{\uparrow}d]} $, we replace $\lsup{j+1}{x}$ with $
\lsup{j}{u} \cdot \lsup{j}{x}$ and reduce cyclically. In passing
from
 $[\mu_{j+1}{\uparrow}\nu_{j+1}]$ to $[\mu_{j}{\uparrow}\nu_{j}]$, we may
add $j$, we may delete some values, and we then take the convex hull
in $\Z$. By repeating this change sufficiently often, we find that,
for $j<<0$, $j \le \mu_j$.  Among all $j \in \Z$ such that $j \le
\mu_j$, let us choose one that minimizes $\nu_j - \mu_j \in
[0{\uparrow}\infty[$.

To pass from $(\,\lsup{j}{x}) \,\,\vee \,\,\lsup{\Z}{z}_{[1{\uparrow}d]} $
to $(\,\lsup{j+1}{x}) \,\,\vee \,\, \lsup{\Z}{z}_{[1{\uparrow}d]} $,
we replace $\lsup{j}{x}$
with $ \lsup{j}{\overline u} \cdot \lsup{j+1}{x}$ and reduce cyclically.
In passing from $[\mu_j{\uparrow}\nu_j] \subseteq
[j{\uparrow}\infty[$ to $[\mu_{j+1}{\uparrow}\nu_{j+1}]$, we may add
$j$, we may delete some values, and we then take the convex hull in
$\Z$. Hence, $[\mu_{j+1}{\uparrow}\nu_{j+1}] \subseteq
[j{\uparrow}\nu_j]$; that is,
 $\mu_{j+1} \ge j$ and $\nu_{j+1} \le \nu_j$.

If $\mu_{j+1} = j$, we take $k = j$.

To prove the claim, it remains to consider the case where
$\mu_{j+1} \ge j+1$.  By the minimality assumption, $\mu_{j+1}
\le \mu_j$.  Hence, $j+1 \le \mu_{j+1} \le \mu_j$. Since  $j+1
\le \mu_j$, when we replace  $\lsup{j}{x}$ with $
\lsup{j}{\overline u} \cdot \lsup{j+1}{x}$ the occurrences of
$\lsup{j}{\overline u}$ survive unaffected by (cyclic) reduction.
Since $j+1 \le \mu_{j+1}$, we see that $\lsup{j}{\overline u} = 1$.
Hence $u=1$ and, hence, $\mu_j$ does not depend on $j$.  We take
$k = \mu_0$, and then $\mu_{k+1} = \mu_0 = k$.

In all cases, then, we have some $k \in \Z$ such that $\mu_{k+1} = k$.  By replacing $r$
with $\lsup{y^{-\!k}}\!{r}$,
we may arrange that $k = 0$.   Now $r$ has become a Hempel relator.
Notice that $r$ is not conjugate to a power of $\,\lsup{0}{\overline u} \cdot \lsup{1}{x}$ in
$ \gp{ (\,\lsup{1}{x})\,\, \vee \,\,\lsup{[0{\uparrow}\infty[\,}{z}_{[1{\uparrow}d]} }{\quad}$
because
$r \text{\normalfont \,mod}[x,y]u$ is not conjugate to a power of $\lsup{0}{x}$
in $ \gp{ (\,\lsup{1}{x})\,\, \vee \,\,\lsup{\Z}{z}_{[1{\uparrow}d]} }{\quad}$.
\end{proof}

In Examples~\ref{Exs:vor}, we shall see that, in
the case where $r \in \gen{x}$,  $G$ is virtually  one-relator;  we consider such groups
 to be well understood.

\section{HNN decomposition, local indicability and torsion}\label{sec:HNN}

In this section, we shall extend three types of results of Hempel and Howie,
namely the
HNN decomposition, the local indicability, and the analysis of torsion.

\begin{Not} We shall use a left-right twisting of the notation
in~\cite[Examples~I.3.5(v)]{DicksDunwoody89},
and write $G_{v}\mathop{\ast}\limits_{G_e} t_e$ to denote an HNN extension,
where it is understood that $G_v$ is a group,
$G_e$~is a  subgroup  of~$G_v$,
there is specified some injective homomorphism
$\overline t_e \colon\!  G_e \to G_v$, $g \mapsto \, \lsup{\overline t_e}{g}$,
and the associated HNN extension is
$$ G_{v}\mathop{\ast}\limits_{G_e} t_e \coloneqq (G_{v}
 \ast \gp{t_e}{\quad})/\normgen{\, \{\, \overline t_e{\cdot}\overline g{\cdot}t_e{\cdot}\lsup{\overline t_e}
g \mid g \in
  G_e\}\,}.$$
If $G = G_{v}\mathop{\ast}\limits_{G_e} t_e \coloneqq G_{v}$,
then the Bass-Serre $G$-tree has vertex set $G/G_v$ and edge set $G/G_e$,
with $gG_e$ joining $gG_v$ to $gt_eG_v$, for each $g \in G$.   \hfill\qed
\end{Not}

In the case where $r$ is a Hempel relator we shall now see that we
get an HNN extension of a one-relator group over a free group.

\begin{Not}\label{Not:HNN}
With {\normalfont Notation~\ref{Not:gen}}, suppose that $r$ is a Hempel relator
for the presentation  $\prs{(x,y)\vee z_{[1{\uparrow}d]}}{[x,y]u}$; thus,
for $X_1\coloneqq  (\,\lsup{1}{x})\,\, \vee\,\,  \lsup{[0{\uparrow}\infty[\,}{z}_{[1{\uparrow}d]} $,
the following hold.
\begin{enumerate} [(R1)]
\vskip-0.7cm \null
\item\label{it:1} $r  \in \gp{X_1}{\quad} \le N(F) \le F$.
\vskip-0.7cm \null
\item\label{it:2}  In $ \gp{X_1}{\quad}$,
 $r$ is not conjugate to any element of $\gen{x}$, where $x \coloneqq \lsup{0}{\overline u} \cdot \lsup{1}{x}$.
In particular,  $r \ne 1$ and, hence, $d\ne 0$.
\vskip-0.7cm \null
\item \label{it:3}  With respect to $X_1$, $r$~is cyclically reduced.
\vskip-0.7cm \null
\item \label{it:5} With respect to $X_1$,  $r$  involves  some element of~$\lsup{0}{z}_{[1{\uparrow}d]}$.
\vskip-0.7cm \null
\item\label{it:4}  With respect to $X_1$,   $r$ involves some element of $\lsup{\nu}{z}_{[1{\uparrow}d]}$,
where $\nu = \nu(r)$ denotes the least element of $[0{\uparrow}\infty[$\, such that
$r \in \gp{(\,\lsup{1}{x})\,\, \vee\,\,  \lsup{[0{\uparrow}\nu]}{z}_{[1{\uparrow}d]}}{\quad}$.
\end{enumerate}

Since  $\lsup{1}{x} = \lsup{0}{u} \cdot x$, we can
identify
 $\gp{(x)  \,\, \vee  \,\, \lsup{[0{\uparrow}\nu]}{z}_{[1{\uparrow}d]}}{\quad}=
\gp{(\,\lsup{1}{x})  \,\, \vee  \,\, \lsup{[0{\uparrow}\nu]}{z}_{[1{\uparrow}d]}}{\quad}$,
and thus view $r$ as an element of a free group with two specified free generating sets.
With respect to $(x)  \,\, \vee  \,\, \lsup{[0{\uparrow}\nu]}{z}_{[1{\uparrow}d]}$,
$r$ involves some element of $\lsup{\nu}{z}_{[1{\uparrow}d]}$,  even if $\nu = 0$, by~(R\ref{it:4}) and~(R\ref{it:2}).
With respect to $(\,\lsup{1}{x})  \,\, \vee  \,\, \lsup{[0{\uparrow}\nu]}{z}_{[1{\uparrow}d]}$,
 $r$ involves some element of $\lsup{0}{z}_{[1{\uparrow}d]}$, by~(R\ref{it:5}).
We define
\begin{align*}
G_{[0{\uparrow}\nu]} &\coloneqq \gp{(x)  \,\, \vee  \,\, \lsup{[0{\uparrow}\nu]}{z}_{[1{\uparrow}d]}}{r}=
\gp{(\,\lsup{1}{x})  \,\, \vee  \,\, \lsup{[0{\uparrow}\nu]}{z}_{[1{\uparrow}d]}}{r},\\
G_{[0{\uparrow}(\nu-1)]} &\coloneqq \gp{(x)  \,\, \vee  \,\, \lsup{[0{\uparrow}(\nu-1)]}{z}_{[1{\uparrow}d]}}{\quad},\\
G_{[1{\uparrow}\nu]} &\coloneqq \gp{(\,\lsup{1}{x})  \,\, \vee  \,\, \lsup{[1{\uparrow}\nu]}{z}_{[1{\uparrow}d]}}{\quad}.
\end{align*}
By Magnus' Freiheitssatz, which appears as Corollary~\ref{Cor:Magnus} in the appendix,
the natural maps from $G_{[0{\uparrow}(\nu-1)]}$ and
$G_{[1{\uparrow}\nu]}$ to $G_{[0{\uparrow}\nu]}$ are injective.

We have an isomorphism  $y\colon G_{[0{\uparrow}(\nu-1)]} \to G_{[1{\uparrow}\nu]}$
 given by the natural bijection on the specified free generating sets.
We can then form the HNN extension
$G_{[0{\uparrow}\nu]} \mathop{\ast}\limits_{G_{[1{\uparrow}\nu]}} y$.
On simplifying the presentation we recover the presentation of $G$ and thus
obtain the HNN decomposition
$G = G_{[0{\uparrow}\nu]} \mathop{\ast}\limits_{G_{[1{\uparrow}\nu]}} y$.
\hfill\qed
\end{Not}

In the case where $G$ is  a hyperbolic  orientable \ors group,
this HNN decomposition was obtained topologically
by Howie~\cite[Proposition~2.1.2(c)]{Howie04} who attributed it to an argument
implicit in~\cite[Proof of Theorem~2.2]{Hempel90}
which in turn is attributed to Howie.

\begin{Def} A group is said to be \textit{locally indicable} if each
finitely generated subgroup  \textit{either}  is  trivial \textit{or}   has  some infinite, cyclic
quotient. \hfill\qed
\end{Def}

\begin{Thm}\label{Thm:hemp} Let $d \in [0{\uparrow}\infty[$, let
$F\coloneqq \gp{(x,y)\vee z_{[1{\uparrow}d]}}{\quad}$, and let
$u$ and $r$ be elements of~$F$.  Suppose that
$u \in \gen{z_{[1{\uparrow}d]}}$, and that $r$
is a Hempel relator for $\prs{(x,y)\vee z_{[1{\uparrow}d]}}{[x,y]u}$.
Let $G\coloneqq \gp{(x,y)\vee z_{[1{\uparrow}d]}}{[x,y]u, r}$.
If $\sqrt[F]{r} = r$, then $G$ is locally indicable.
\end{Thm}

\begin{proof}  We use Notation~\ref{Not:HNN}.

 In $G$, define, for each  $i \in \Z$,\vspace{-1mm}
\begin{align*}
G_{[i{\uparrow}(i+\nu-1)]} &\coloneqq  \lsup{y^i}\!{(G_{[0,(\nu-1)]})} =
\gp{ (\,\lsup{i}{x})\,\, \vee\,\, \lsup{[i{\uparrow}(i+\nu-1)]}{z}_{[1{\uparrow}d]} }{ \quad },\\
G_{[i{\uparrow}(i+\nu)]} &\coloneqq  \lsup{y^i}\!{(G_{[0,\nu]})} =
\gp{ (\,\lsup{i}{x})\,\, \vee\,\, \lsup{[i{\uparrow}(i+\nu)]}{z}_{[1{\uparrow}d]} }{\,\,\lsup{i}{r}\, }.
\end{align*}

We can then write\vspace{-1mm}
\begin{equation}\label{eq:freeprod2}
G_{[i{\uparrow}(i+\nu)]}  = (G_{[i{\uparrow}(i + \nu-1)]} \ast
\gp{\lsup{i+\nu}{z}_{[1{\uparrow}d]}  }{ \quad })/\normgen{\lsup{i}{r}}.
\end{equation}

Since $\sqrt[F]{r} = r$,  $\lsup{i}{r}$ is not a proper power in $G_{[i{\uparrow}(i + \nu-1)]} \ast
\gp{\lsup{i+\nu}{z}_{[1{\uparrow}d]}  }{ \quad }$.

By~(R\ref{it:2}),~(R\ref{it:3}) and~(R\ref{it:4}), $\lsup{0}{r} \in G_{[0{\uparrow}(\nu-1)]} \ast
\gp{\lsup{\nu}{z}_{[1{\uparrow}d]}  }{ \quad }$ is not conjugate to any
element of $G_{[0{\uparrow}(\nu-1)]}$.  Hence $\lsup{i}{r} \in G_{[i{\uparrow}(i + \nu-1)]} \ast
\gp{\lsup{i+\nu}{z}_{[1{\uparrow}d]}  }{ \quad }$ is not conjugate to any
element of $G_{[i{\uparrow}(i + \nu-1)]}$.

By using $i+ \nu +1$, we can form the free product with amalgamation
\begin{align*}
G_{[i{\uparrow}(i+\nu+1)]} &\coloneqq  G_{[i{\uparrow}(i+\nu)]}
 \mathop{\ast}\limits_{\textstyle
\gp{(\, \lsup{i}{u}\cdot \lsup{i}{x} = \lsup{i+1}{x}) \,\,\,
 \vee \,\,\,  \lsup{[(i+1){\uparrow}(i+\nu)]}{z}_{[1{\uparrow}d]}}{\quad} }
G_{[(i+1){\uparrow}(i+\nu+1)]}\\
&=  G_{[i{\uparrow}(i+\nu)]}
 \mathop{\ast}\limits_{G_{[(i+1){\uparrow}(i+\nu)]}}
G_{[(i+1){\uparrow}(i+\nu+1)]}.
\end{align*}
By varying $i$, we get a bi-infinite chain of free products with amalgamation
$$N(G)  \coloneqq \quad \cdots \,\, G_{[(i-1){\uparrow}(i+\nu-1)]}
 \mathop{\ast}\limits_{\textstyle G_{[i{\uparrow}(i+\nu-1)]}}
G_{[i{\uparrow}(i+\nu)]} \mathop{\ast}\limits_{\textstyle G_{[(i+1){\uparrow}(i+\nu)]}}
G_{[(i+1){\uparrow}(i+\nu+1)]} \,\, \cdots.$$
Here $C_\infty = \gp{y}{\quad}$ acts by shifting and we find that $N(G) \rtimes C_\infty = G$.
For any finite, non-empty interval $[j{\uparrow}i]$ in $\Z$ let us define
\begin{align*}
G_{[j{\uparrow}(i+\nu)]} \coloneqq \quad
G_{[j{\uparrow}(j+\nu)]}
 \mathop{\ast}\limits_{\textstyle G_{[(j+1){\uparrow}(j+\nu)]}} \cdots
 \mathop{\ast}\limits_{\textstyle G_{[i{\uparrow}(i+\nu-1)]}}
G_{[i{\uparrow}(i+\nu)]},
\end{align*}
a subgroup of $G$.  By using~\eqref{eq:freeprod2}, we see that
\begin{equation*}
G_{[j{\uparrow}(i+\nu)]} = (G_{[j{\uparrow}(i + \nu-1)]} \ast
\gp{ \lsup{i+\nu}{z}_{[1{\uparrow}d]} }{\quad})/\normgen{\lsup{i}{r} }.
\end{equation*}
 Now
$\lsup{i}{r} \in G_{[i{\uparrow}(i+\nu-1)]} \ast
\gp{ \lsup{i+\nu}{z}_{[1{\uparrow}d]} }{\quad} \subseteq G_{[j{\uparrow}(i+\nu-1)]}\vspace{1mm}\ast
\gp{ \lsup{i+\nu}{z}_{[1{\uparrow}d]} }{\quad}$,
and $\lsup{i}{r}$ has the same cyclically reduced expression in both free products.  In particular,
$\lsup{i}{r}$  is not a proper power and is not conjugate to any
element of $G_{[j{\uparrow}(i + \nu-1)]}\vspace{1mm}$.
By a result of Howie given as Corollary~\ref{Cor:redpres} in the appendix,  the subgroup
$G_{[j{\uparrow}\infty[} \,\coloneqq \bigcup\limits_{i\in[j{\uparrow}\infty[}G_{[j{\uparrow}(i+\nu)]}$
is  then locally indicable.
It follows that
$\bigcup\limits_{j \in [0{\downarrow}(-\infty)[}G_{[j{\uparrow}\infty[}\vspace{1mm}$
is locally indicable, that is, $N(G)$ is locally indicable.
Hence, $N(G) \rtimes C_\infty$ is locally indicable, that is,
$G$ is locally indicable.
This completes the proof.
\end{proof}

In the case where $G$ is a hyperbolic orientable \ors group,  the above result
 was given by Hempel~\cite[Theorem~2.2]{Hempel90} with its proof attributed to Howie.

We now discuss torsion.

\begin{Thm} \label{thm:torsion} Let $d \in [0{\uparrow}\infty[$,
let $F\coloneqq  \gp{(x,y)\vee z_{[1{\uparrow}d]}}{\quad}\vspace{1mm}$,
and let $u$ and $r$ be elements of~$F$.  Suppose that
$u \in \gen{z_{[1{\uparrow}d]}}\vspace{1mm}$ and that $r$
is a Hempel relator for $\prs{(x,y)\vee z_{[1{\uparrow}d]}}{[x,y]u}$.
Let $G\coloneqq \gp{(x,y)\vee z_{[1{\uparrow}d]}}{[x,y]u, r}\vspace{1mm}$.

Let  $m\coloneqq \log_{F}r$  and let
$C_m \coloneqq \gen{\sqrt[F]{r}\,}/\normgen{r}$.
Then the following hold.
\begin{enumerate}[\normalfont (i).]
\vskip-0.8cm \null
\item $C_m$ can be identified with  the subgroup of $G$ generated by the image of $\sqrt[F]{r}$.
\vskip-0.8cm \null
\item $G/\gen{\lsup{G}{C_m}} = \gp{(x,y)\vee z_{[1{\uparrow}d]}}{[x,y]u, \sqrt[F]{r}\,}$
which is locally indicable.
\vskip-0.8cm \null
\item  There exists a subset $X$ of $G$ such that
$\gen{\lsup{G}{C_m}} = \mathop{\text{\Large$\ast$}}(\lsup{X}{C_m})
\coloneqq\mathop{\text{\Large$\ast$}}\limits_{g \in X}(\lsup{g}{C_m})$.
  Let $K$ denote  the kernel
of the homomorphism $\mathop{\text{\Large$\ast$}} (\lsup{X}{C_m}) \to C_m$
which acts as conjugation by $\overline g$ on $\lsup{g}{C_m}$, for each $g \in X$.
Then $K$ is a free group, and   $\gen{\lsup{G}{C_m}} = K \rtimes C_m$.
\vskip-0.8cm \null
\item\label{it:12} Each torsion subgroup of $G$ lies in some conjugate of $C_m$.
\vskip-0.8cm \null
\item Every torsion-free subgroup of $G$ is locally indicable.
\vskip-0.8cm \null
\item\label{it:10} $G$ has some torsion-free finite-index subgroup.
\end{enumerate}
\end{Thm}

\begin{proof} We again  use Notation~\ref{Not:HNN}.

Notice that $d \ne 0$ and $r \ne 1$, by~(R\ref{it:2}).

Notice that $\sqrt[F]{r}$ too is a Hempel relator for $\prs{(x,y) \vee z_{[1{\uparrow}d]}}{[x,y]u}$,
and that $\nu(\!\!\sqrt[F]{r}){=}\nu(r)$.

(i).
By results of Magnus, we can view $C_m$ as a subgroup of the one-relator group
 $G_{[0{\uparrow}\nu]}$.
By results of Higman-Neumann-Neumann,
we can view $G_{[0{\uparrow}\nu]}$ as a subgroup of $G$. Thus (i) holds.

(ii) follows easily from Theorem~\ref{Thm:hemp}.

(iii).  It is straightforward to see the following.\vspace{-2mm}
\begin{align*}
G/\gen{\lsup{G}{C_m}} &= \gp{(x,y) \vee z_{[1{\uparrow}d]}}{[x,y]u, \sqrt[F]{r}\,}\\
&= \gp{(x)   \vee  \, \lsup{[0{\uparrow}\nu]}{z}_{[1{\uparrow}d]}}{\sqrt[F]{r}\,}
 \mathop{\ast}\limits_{\gp{ (\lsup{1}{x}) \vee\, \lsup{[1{\uparrow} \nu  ]}{z}_{[1{\uparrow}d]}}{\quad}} y \\[-.1cm]
& = (G_{[0{\uparrow}\nu]}/\gen{\lsup{G_{[0{\uparrow}\nu]}}{C_m}})\mathop{\ast}\limits_{G_{[1{\uparrow}\nu]}} y.
\end{align*}

\noindent If we consider $\gen{\lsup{G}{C_m}}$
 acting on the Bass-Serre tree for the HNN decomposition of $G$,
we now see that $\gen{\lsup{G}{C_m}}$  acts freely on the edge set,
and that the quotient graph is the Bass-Serre tree for the above HNN decomposition of $G/\gen{\lsup{G}{C_m}}$.
Hence $\gen{\lsup{G}{C_m}}$ is a free product of conjugates of $\gen{\lsup{G_{[0{\uparrow}\nu]}}{C_m}}\vspace{1mm}$,
with one conjugate for each vertex of
the latter tree.   By~\cite[Theorem~1]{FKS}, $\gen{\lsup{G_{[0{\uparrow}\nu]}}{C_m}}$
in turn is a free product of conjugates of $C_m$.
Hence, $\gen{\lsup{G}{C_m}}$ is a free product of conjugates of $C_m$.
Hence we have the desired homomorphism $\gen{\lsup{G}{C_m}} \to C_m$, and its kernel  $K$.
Clearly $\gen{\lsup{G}{C_m}} = K \rtimes C_m$.
If we consider $K$ acting on the Bass-Serre tree associated with the
graph-of-groups decomposition of $\gen{\lsup{G}{C_m}}$ as a free product of
copies  of~$C_m$, we see that $K$ acts freely on the vertex set, and, hence, $K$ is a free group.

(iv).  Suppose that $H$ is some torsion subgroup of $G$.
The image of $H$ in the torsion-free quotient   $G/\gen{\lsup{G}{C_m}}$ is then trivial,
that is, $H \le \gen{\lsup{G}{C_m}}=  K \rtimes C_m$.
Since $H \cap K$ is necessarily trivial, $H$ embeds in $C_m$ and, in particular, $H$ is finite.
Also, $H$
 lies in $\gen{\lsup{G}{C_m}}= \mathop{\text{\Large$\ast$}} (\lsup{X}{C_m})$.
By Bass-Serre theory, or the Kurosh subgroup theorem,
$H$ lies in some conjugate of $C_m$.

(v). Suppose that $H$ is some torsion-free subgroup of $G$.
Then, by Bass-Serre theory, or the Kurosh subgroup theorem,
$H \cap (\mathop{\text{\Large$\ast$}} (\lsup{X}{C_m})) = H \cap  \gen{\lsup{G}{C_m}}$ is free, and,
hence,  locally indicable.  Now $H/(H\cap\gen{\lsup{G}{C_m}})$ embeds in
the locally indicable group $G/\gen{\lsup{G}{C_m}}$, and hence is locally indicable.
It follows that $H$ is locally indicable.

(vi). We imitate the proof of~\cite[Theorem~2]{FKS}.
By Lemma~\ref{Lem:Allenby}, or Remark~\ref{Rem:Allenby}, there exists some
finite group $\Phi$ and some homomorphism $\alpha \colon \gp{(x,y)\vee z_{[1{\uparrow}d]}}{[x,y]u} \to \Phi$ such that
$\alpha(\!\sqrt[F]{r}\,)$ has order exactly $m$.
This induces a homomorphism
$$\beta \colon \gp{(x,y)\vee z_{[1{\uparrow}d]}}{[x,y]u, r} \to \Phi$$
which is injective on $C_m$.  Let $N$ denote the kernel of $\beta\colon G \to \Phi$.
If $H$ is some torsion subgroup of~$N$, then, by~(iv),  $H \subseteq \lsup{g}{C_m} \cap N$ for some $g \in G$.
Since $\beta$ is injective on $\lsup{g}{C_m}$ and vanishes on $N$,
we see that $\lsup{g}{C_m}  \cap N = \{1\}$. Thus $N$ is torsion free and of finite index in $G$.
\end{proof}

In the case where $G$ is a hyperbolic orientable \ors group,
Howie~\cite{Howie04} proved several of these results.

\section{Exact sequences}\label{sec:exact}

In this section, we construct an exact sequence that shows the existence of a
two-dimensional $\uE G$ when we have a Hempel presentation.

\begin{Rem}\label{Rem:total}  We shall find ourselves considering diagrams of  abelian groups
of the forms that appear in Fig~\ref{fig:0},  where maps are written on the
right of their arguments.
\begin{figure}[t!]
\begin{center}
$\begin{CD}
 && &&  0\\
 && && @VVV\\
  &&0 &&D\\
  &&@VVV @VV{f}V\\
&& A @>c>> E\\
&& @V{a}VV @VV{g}V\\
&& B @>d>> F\\
&& @V{b}VV @VV{h}V\\
0 @>>> C @>{e}>> G  @>i>> H @>>> 0\\
&& @VVV @VVV\\
&& 0 &&0
\end{CD}
\qquad \qquad
\begin{CD}
 0\\
@VVV\\
  A\oplus D\\
  @VV{\left(\begin{smallmatrix} -a&c\\0&f
\end{smallmatrix}\right)}V\\
B\oplus E\\
@VV{\left(\begin{smallmatrix} d\\g
\end{smallmatrix}\right)}V\\
F\\
 @VV{hi}V \\
H\\
 @VVV\\
0
\end{CD}$
\end{center}
\caption{A double complex and its total complex}\label{fig:0}
\end{figure}
If the left-hand diagram is commutative and
its two columns and  long row  are exact,
then the right-hand diagram is an exact sequence; this implication follows from the
theory of total complexes of double complexes, and is easy to check by chasing diagrams.
\hfill\qed
\end{Rem}

\begin{Not}\label{Not:Fox}  Let $W$ be a set and let $G$ be a group.

We identify $\Z[G{\times}W] = \Z G \otimes_\Z \Z W $,
and, hence, for each $(p,w)\in \Z G \times W$,
we view $p\otimes w$ as an element of  $\Z[G{\times}W]$.

We also identify $\Z[G{\times}W]$ with the direct sum of a family of copies of $\Z G$ indexed by~$W$,
denoted $(\Z G)^W$.

For any subgroup $C$ of $G$,
we consider $C$ as acting trivially on $W$ and write $\Z[G \times_C W] = \Z G \otimes_{\Z C} \Z W $.

For any map of sets   $\alpha\colon W \to G$,   $w \mapsto  \alpha(w)$, the map of sets
$$W \to (\Z[G{\times}W]){\rtimes}G = \begin{pmatrix}
1&0\\ \Z[G{\times}W] &G
\end{pmatrix}, \quad w \mapsto ( 1\otimes w, \alpha(w)) = \begin{pmatrix}
1&0\\
1\otimes w &\alpha(w)
\end{pmatrix},$$ induces
a unique group homomorphism $\gp{W}{\quad} \to   (\Z[G{\times}W]) \rtimes G$, denoted
$r \mapsto (\frac{\partial r}{\partial W}, \alpha(r))$.  We call $\frac{\partial r}{\partial W}$
the \textit{total Fox derivative} of $r$ (with respect to $W$, relative to $\alpha$).
If the map $\alpha$ is understood, then, for each $w \in W$,  we  write
$\frac{\partial r}{\partial w}  \otimes w$ for the summand
of $\frac{\partial r}{\partial W}$ corresponding to~$w$.
\hfill\qed
\end{Not}

\newpage

\begin{Thm}\label{Thm:exact} Let $d \in [0{\uparrow}\infty[$\,, let
$F\coloneqq  \gp{(x,y)\vee z_{[1{\uparrow}d]}}{\quad}\vspace{1mm}$,
and let $u$ and $r$ be elements of~$F$.  Suppose that
$u \in \gen{z_{[1{\uparrow}d]}}$ and that $r$
is a Hempel relator for $\prs{(x,y)\vee z_{[1{\uparrow}d]}}{[x,y]u}$.
Let $G\coloneqq \gp{(x,y)\vee z_{[1{\uparrow}d]}}{[x,y]u, r}\vspace{1mm}$.

 Let $C$ denote the subgroup of $G$ generated by the image of $\sqrt[F]{r}$.
Let $(x,y,z_{[1{\uparrow}d]})\vspace{1mm}$ denote $(x,y) \vee z_{[1{\uparrow}d]}\vspace{1mm}$,
and let $ G \times ([x,y]u, \null_Cr)$ denote $(G \times \{[x,y]u\}) \,\, \cup (G \times_C \{r\})$.
Then the  sequence of left $\Z G$-modules given by
\begin{equation*}
0 {\to} \Z[G \times ([x,y]u, \null_Cr)]{\xrightarrow{ \left(\begin{array}{lll}
\scriptstyle 1\otimes [x,y]u \\ \scriptscriptstyle \mapsto
 \frac{ \partial [x,y]u}{\partial(x,y,z_{[1{\uparrow}d]})} \\[.3cm]
\scriptstyle 1\otimes_C r \\ \scriptscriptstyle \mapsto \frac{\partial r}{\partial (x,y,z_{[1{\uparrow}d]})}
\end{array}\right)}} \Z[G \times (x,y,z_{[1{\uparrow}d]})] {\xrightarrow{
\hskip-.15cm\scriptscriptstyle\left(\begin{array}{l}
\scriptstyle 1\otimes w \\ \scriptstyle \mapsto (w-1)1
\end{array}\scriptscriptstyle\right)
}}
\Z[G/1] {\xrightarrow{\left(\scriptstyle 1 \mapsto G\right)}} \Z[G/G] {\to} 0
\end{equation*}
is exact.
\end{Thm}

\begin{Rem} The exact sequence of left $\Z G$-modules in Theorem~\ref{Thm:exact}, which has the form
\begin{equation*}
0 \to \Z G \oplus \Z [G/C]
\xrightarrow{
}
\Z G^{d+2}
\xrightarrow{  }
\Z G \xrightarrow{} \Z \to 0,
\end{equation*}
is the augmented cellular chain
complex of the `adjusted' universal cover of the Cayley complex of the presentation
$$\prs{(x,y) \vee z_{[1{\uparrow}d]}}{[x,y]u,   r},$$
where `adjusted' means the following.
In the universal cover, at each zero-cell $g \in G$,
there is attached a two-cell $g\mathbf{r}$ whose boundary reads representatives of
$\lsup{g}{r}$ in $F$ off the one-skeleton.
 Let $c$ denote the image of  $ \sqrt[F]{r}$ in $G$.
Then $c^m=1$ and the boundaries of $gc\mathbf{r}$ and $g\mathbf{r}$  read the same elements of~$F$;
 if $m \ge 2$, this
 creates a two-sphere in the universal cover.  By  `adjusting'  the universal
cover we mean
identifying $g\mathbf{r} = gc\mathbf{r}$ for each $g \in G$.
To make the $G$-action cellular, we should subdivide $\mathbf{r}$ into $m$ regions permuted by
$c$ with a single point fixed by~$c$.
This would allow us to divide out by the action of $G$ to recover an
`adjusted' Cayley complex, a two-dimensional  $\uE G$, which has a two-cell $\mathbf{r'}$,
whose boundary reads $\sqrt[F]{r}$, with an interior $\frac{1}{m}\vspace{1mm}$th-point,
or  cone-point of angle~$\frac{2\pi}{m}$,
which means that a path travelling $m$ times around the boundary of  $\mathbf{r'}$ is  null-homotopic.
\hfill\qed
\end{Rem}

\begin{proof}[Proof of Theorem~\ref{Thm:exact}]
In\vspace{1mm} Notation~\ref{Not:HNN}, let $(x, \lsup{[0{\uparrow}\nu]}{z}_{[1{\uparrow}d]})$ denote
$(x) \,\,\vee \,\, \lsup{[0{\uparrow}\nu]}{z}_{[1{\uparrow}d]}$,
let\newline $(x, \lsup{[1{\uparrow}\nu]}{z}_{[1{\uparrow}d]})$ denote
$(x) \,\,\vee \,\, \lsup{[1{\uparrow}\nu]}{z}_{[1{\uparrow}d]}$, and let
\begin{align*}
&\gen{x, \lsup{[0{\uparrow}\nu]}{z}_{[1{\uparrow}d]}}  \coloneqq \,\,G_{[0{\uparrow}\nu]}  \coloneqq \,\,
 \gp{(x) \,\,\vee\,\, \lsup{[0{\uparrow}\nu]}{z}_{[1{\uparrow}d]}}{r},  \\[.3cm]
&\gen{\,\lsup{1}{x}, \lsup{[1{\uparrow}\nu]}{z}_{[1{\uparrow}d]}}  \coloneqq\,\, G_{[1{\uparrow}\nu]}
 \coloneqq \,\, \gp{(x)\,\, \vee \,\,\lsup{[1{\uparrow}\nu]}{z}_{[1{\uparrow}d]}}{\quad}.
\end{align*}
Recall that $y \colon \gen{x, \lsup{[0{\uparrow}(\nu-1)]}{z}_{[1{\uparrow}d]}} \to
 \gen{\,\lsup{1}{x}, \lsup{[1{\uparrow}\nu]}{z}_{[1{\uparrow}d]}}$
acts by $
x\mapsto  \lsup{1}{x}$, $\lsup{i}{z}_{\ast} \mapsto \lsup{i+1}{z}_{\ast}$,
and  that\linebreak $G =\gen{x, \lsup{[0{\uparrow}\nu]}{z}_{[1{\uparrow}d]}}
 \mathop{\ast}\limits_{\gen{\lsup{1}{x},\lsup{[1{\uparrow}\nu]}{z}_{[1{\uparrow}d]}} } y$.
All the notation involved in Fig.~\ref{fig:1} has now been explained.

\begin{figure}[t!]
\begin{center}
$\begin{CD}
\\
@.
0
\\[.3cm]
@.
@VVV
\\[.3cm]
0
@.
\Z[G/C]
\\[.3cm]
@VVV
@VV{\scriptstyle\left(\begin{array}{l}
\scriptstyle C\,\,   \mapsto  \\
\scriptstyle\frac{\partial r(x,\lsup{[0{\uparrow}\nu]}{z}_{[1{\uparrow}d]})}
{\partial(x,\lsup{[0{\uparrow}\nu]}{z}_{[1{\uparrow}d]})}
\end{array}\scriptstyle\right)}V
\\[-3.3cm]
\Z[G \times (\,\lsup{1}{x},\lsup{[1{\uparrow}\nu]}{z}_{[1{\uparrow}d]})]
@>{\scriptstyle\left(\begin{array}{l}
\scriptstyle 1 \otimes\, \lsup{1}{x}   \,\, \mapsto  \\
   \scriptstyle (y-\lsup{0}{u})\otimes x  \\  \scriptstyle -
\frac{\partial\,\lsup{0}{u}}{\partial \,\lsup{0}{z}_{[1{\uparrow}d]}}\\ \\
\scriptstyle 1\otimes \,\,\lsup{i}{z}_{\ast} \,\, \mapsto \\
\scriptstyle y \otimes \,\,\lsup{i-1}{z}_{\ast}  \\\scriptstyle -  1\otimes \,\,\lsup{i}{z}_{\ast}
\end{array}\scriptstyle\right)
} >\phantom{xxxxxxxxxxxxxxxxx}>
\Z[G \times (x,\lsup{[0{\uparrow}\nu]}{z}_{[1{\uparrow}d]})]
\\[.3cm]
@VV{\scriptstyle \left(\begin{array}{lll}
\scriptstyle 1\otimes w   \,\,  \mapsto   \\ \scriptstyle (w-1)\otimes y
\end{array}\scriptstyle \right)}V
@VV{\scriptstyle\left(\begin{array}{l}
\scriptstyle 1\otimes w   \,\, \mapsto \\ \scriptstyle (w-1)1
\end{array}\scriptstyle\right)}V
\\[-.5cm]
\Z[G \times (y)]
@>{\scriptstyle\left(\begin{array}{l}
\scriptstyle 1\otimes y  \,\, \mapsto  \\ \scriptstyle (y-1)1
\end{array}\scriptstyle\right)}>\phantom{xxxxxxxxxxxxxxxxx}>
\Z[G/1]
\\[.5cm]
@VV{\scriptstyle\left(\begin{array}{l}
\scriptstyle 1\otimes y  \,\, \mapsto \\
\scriptstyle \gen{\,\lsup{1}{x},\lsup{[1{\uparrow}\nu]}{z}_{[1{\uparrow}d]}}
\end{array}\scriptstyle\right)}V
@VV{\scriptstyle\left(\begin{array}{l}
\scriptstyle  1\,\, \mapsto \\  \scriptstyle \gen{x,\lsup{[0{\uparrow}\nu]}{z}_{[1{\uparrow}d]}}
\end{array}\scriptstyle\right)}V
\\[.3cm]
0 \to \Z[G/\gen{\,\lsup{1}{x},\lsup{[1{\uparrow}\nu]}{z}_{[1{\uparrow}d]}}]
@>{\scriptstyle\left(\begin{array}{l}
\scriptstyle \gen{\,\lsup{1}{x},\lsup{[1{\uparrow}\nu]}{z}_{[1{\uparrow}d]}} \mapsto  \\ \scriptstyle
(y-1)\gen{ x,\lsup{[0{\uparrow}\nu]}{z}_{[1{\uparrow}d]}}
\end{array}\scriptstyle\right)}>\phantom{xxxxxxxxxxxxxxxxx}>
\Z[G/\gen{x,\lsup{[0{\uparrow}\nu]}{z}_{[1{\uparrow}d]}}]
@>>>
\Z[G/G] \to 0
\\[.3cm]
@VVV @VVV
\\[.3cm]
0 @. 0
\end{CD}$
\end{center}
\caption{A commuting diagram.}\label{fig:1}
\end{figure}

We now make five observations about Fig.~\ref{fig:1}.
\newline\indent \phantom{ii}(i). The long row is
the exact augmented cellular chain complex of the
Bass-Serre tree corresponding to the HNN graph-of\d1groups
decomposition of $G$ with one vertex and one edge.
 See, for example,~\cite[Examples~I.3.5(v) and Theorem~I.6.6]{DicksDunwoody89}.
\newline\indent\phantom{i}(ii). The left column   is the exact sequence obtained by applying the exact functor
$\Z G\otimes_{\Z[\gen{\,\lsup{1}{x}, \lsup{[1{\uparrow}\nu]}{z}_{[1{\uparrow}d]}}]} (\quad)$
to the exact   augmented cellular chain complex of the
Cayley tree of the free group
$\gen{\,\lsup{1}{x}, \lsup{[1{\uparrow}\nu]}{z}_{[1{\uparrow}d]}}$, or, equivalently,
the Bass-Serre tree corresponding to the graph-of-groups decomposition  with one vertex and
$1 + \nu d$ edges.  See, for example,~\cite[Examples~I.3.5(i) and Theorem~I.6.6]{DicksDunwoody89}.
\newline\indent(iii). The right column  is the exact sequence obtained by applying the exact functor
$\Z G\otimes_{\Z[\gen{x, \lsup{[0{\uparrow}\nu]}{z}_{[1{\uparrow}d]}}]} (\quad)$
to Lyndon's exact sequence  for the one-relator group
$\gp{x, \lsup{[0{\uparrow}\nu]}{z}_{[1{\uparrow}d]}}{r}$;
see, for example,~\cite[($*$) on
 p.~167]{Chiswell03}.
\newline\indent\phantom{i}(iv). It is clear that the lower square commutes.
\newline\indent\phantom{ii}(v). To see that the upper square commutes, we note that along
the upper route in the upper square,
\allowdisplaybreaks
\begin{align*}
1 \otimes \,\lsup{1}{x} \quad \mapsto \quad
  &(y-\lsup{0}{u})\otimes x \,\,-\textstyle\sum\limits_{w \in\,\, \lsup{0}{z}_{[1{\uparrow}d]} }
  \frac{\partial\,\,  \lsup{0}{u}}{\partial w}\otimes w \\
\mapsto \quad   &(y- \lsup{0}{u})(x-1) -  (\lsup{0}{u} -1) = yx  - \lsup{0}{u}x - y +1
 \\&= \lsup{1}{x}\cdot y- \lsup{1}{x} -y+1 = (\,\lsup{1}{x}-1)(y-1),
\intertext{and}
1\otimes \,\lsup{i}{z}_{*}  \quad  \mapsto \quad &y \otimes \,\lsup{i-1}{z}_{*} -  1\otimes \,\lsup{i}{z}_{*}\\
\mapsto \quad  &y(\,\lsup{i-1}{z}_{*} -1) - (\,\lsup{i}{z}_{*} -1) = y\cdot \lsup{i-1}{z}_{*} -y-\lsup{i}{z}_{*} +1
\\&  = \lsup{i}{z}_{*} \cdot y -y - \lsup{i}{z}_{*} + 1 = (\,\lsup{i}{z}_{*} -1)(y-1).
\end{align*}
It is now clear that the upper square commutes.

\begin{figure}[t!]
\begin{center}
$\begin{CD}
0&& 0
\\
@VVV@VVV
\\[0cm]
\Z[G \times (\lsup{1}{x},\lsup{[1{\uparrow}\nu]}{z}_{[1{\uparrow}d]})]  \oplus \Z[G/C]
 @>{\left(\begin{smallmatrix}
1 \otimes w  &\mapsto &   1 \otimes  w\\ \\
C &\mapsto& 1 \otimes_C r
\end{smallmatrix}\right)}>{\phantom{xxxxxxxxxxxxxxxxxxxxxxxxxxxxx}}>
\Z[G \times (\,\lsup{1}{x},\,\lsup{[1{\uparrow}\nu]}{z}_{[1{\uparrow}d]}, \null_Cr)]
\\[0.5cm]
@VV{\left(\begin{smallmatrix}
1 \otimes\,\lsup{1}{x}  & \mapsto  &  (1-\, \lsup{1}{x})\otimes y   \\&&   +(y-\,\lsup{0}{u})\otimes x
\\&&- \frac{\partial \,\lsup{0}{u}}{\partial \,\lsup{0}{z}_{[1{\uparrow}d]}}\\ \\
1\otimes \,\lsup{i}{z}_\ast  &\mapsto & (1-\,\lsup{i}{z}_\ast) \otimes y\\&& + y \otimes\,\lsup{i-1}{z}_\ast
\\&&-  1\otimes \,\lsup{i}{z}_\ast \\ \\
C &\mapsto&\frac{\partial (r(x,\lsup{[0{\uparrow}\nu]}{z}_{[1{\uparrow}d]}))}
{\partial(x,\lsup{[0{\uparrow}\nu]}{z}_{[1{\uparrow}d]})}
\end{smallmatrix}\right)
\hskip 1.2cm=}V
@V{\left(\begin{smallmatrix}
-[x,y] \otimes \,\lsup{1}{x}  & \mapsto  &  \frac{\partial ([x,y]\,\lsup{0}{u})}
                                           {\partial(x,y,\,\lsup{[0{\uparrow}\nu]}{z}_{[1{\uparrow}d]})}
\\ \\ \\  \\ \\ \\[-.3cm]
1\otimes \,\lsup{i}{z}_\ast  &\mapsto &
\frac{\partial (y\cdot\lsup{i-1}{z}_\ast\cdot \overline y  \cdot \lsup{i}{\overline z}_\ast)}
      {\partial(x,y,\,\lsup{[0{\uparrow}\nu]}{z}_{[1{\uparrow}d]})} \\ \\ \\ \\
1 \otimes_C r &\mapsto&\frac{\partial (r(x,\lsup{[0{\uparrow}\nu]}{z}_{[1{\uparrow}d]} ))}
                            {\partial(x,y,\,\lsup{[0{\uparrow}\nu]}{z}_{[1{\uparrow}d]})}
\end{smallmatrix}\right)}VV
\\[.5cm]
\Z[G \times  (y) ] \oplus \Z[G \times (x,\lsup{[0{\uparrow}\nu]}{z}_{[1{\uparrow}d]})]
@=
\Z[G \times (x,y,\,\lsup{[0{\uparrow}\nu]}{z}_{[1{\uparrow}d]})]
\\[.5cm]
@VV{\left(\begin{smallmatrix}
1\otimes w  &\mapsto &(w-1)1\end{smallmatrix}\right)\hskip2.9cm =}V
@V{\left(\begin{smallmatrix}
1\otimes w  &\mapsto (w-1)1
\end{smallmatrix}\right)
}VV
\\[.2cm]
\Z[G/1] @= \Z[G/1]
\\[.2cm]
@VV{(1\,\,\mapsto\,\, G)}V
@V{(1\,\,\mapsto\,\, G)}VV
\\[.2cm]
\Z[G/G] @= \Z[G/G]
\\[.2cm]
@VVV @VVV
\\[.2cm]
0 &&0
\end{CD}$
\end{center}
\caption{An exact sequence rewritten.}\label{fig:2}
\end{figure}

With these five observations in mind, we can
 apply Remark~\ref{Rem:total} to Fig.~\ref{fig:1}, and what we get is
the exact sequence  which appears as the left column of Fig.~\ref{fig:2}.
After some adjustments, it becomes the right column of Fig.~\ref{fig:2}, which can
be viewed as the augmented cellular chain
complex of the adjusted  universal cover of the  Cayley complex of the presentation
$$\prs{(x,y) \,\,\,\vee\,\, \,\lsup{[0{\uparrow}\nu]}{z}_{[1{\uparrow}d]}}{[x,y]u,\,\,
(y\cdot\lsup{i-1}{z}_\ast\cdot \overline y  \cdot \lsup{i}{\overline z}_\ast \mid
z_\ast^{(i)} \in \,\lsup{[1{\uparrow}\nu]}{z}_{[1{\uparrow}d]}),\,\,
r(x,\lsup{[0{\uparrow}\nu]}{z}_{[1{\uparrow}d]}) }.$$

If $\nu = 0$, we have the desired exact sequence.

If $\nu \ge 1$,  then
 we  shall delete $\lsup{[1{\uparrow}\nu]}{z}_{[1{\uparrow}d]}$
from the set of generators and delete
$$(y\cdot\lsup{i-1}{z}_\ast\cdot \overline y  \cdot \lsup{i}{\overline z}_\ast \mid
z_\ast^{(i)} \in \,\lsup{[1{\uparrow}\nu]}{z}_{[1{\uparrow}d]})$$
 from the set of relators,  and  understand, henceforth, that
$ \lsup{i}{z}_\ast$ denotes  $\lsup{y^i}{\!z}_\ast$ in the
new generators.  We consider $G$ as being unaltered, but we alter
the exact sequence by dividing out  by the exact subcomplex
$$0 \to \Z[G \times (\lsup{[1{\uparrow}\nu]}{z}_{[1{\uparrow}d]})])]
\xrightarrow{\begin{pmatrix}
\hskip-.9cm\scriptstyle  1\otimes \,\lsup{i}{z}_\ast  \mapsto
\\[.2cm] \scriptstyle\frac{\partial (y\cdot\lsup{i-1}{z}_\ast\cdot \overline y  \cdot \lsup{i}{\overline z}_\ast)}
      {\partial(x,y,\,\lsup{[0{\uparrow}\nu]}{z}_{[1{\uparrow}d]})}
\end{pmatrix}}
\Z[G \times (\scriptstyle\frac{\partial (y\cdot\lsup{i-1}{z}_\ast\cdot \overline y  \cdot \lsup{i}{\overline z}_\ast)}
      {\partial(x,y,\,\lsup{[0{\uparrow}\nu]}{z}_{[1{\uparrow}d]})}\textstyle\mid \,\lsup{i}{z}_\ast
\in \lsup{[1{\uparrow}\nu]}{z}_{[1{\uparrow}d]})]
\to 0 \to 0 \to 0.$$
Thus, for each $i \in [1{\uparrow}d]$, in the quotient of
$\Z[G \times (\,\lsup{1}{x},\,\lsup{[1{\uparrow}\nu]}{z}_{[1{\uparrow}d]}, \null_Cr)]$, we are identifying
$1\otimes \,\lsup{i}{z}_\ast$ with $0$, while, in the quotient of
$\Z[G \times (x,y,\,\lsup{[0{\uparrow}\nu]}{z}_{[1{\uparrow}d]})]$, we are identifying
$1\otimes \,\lsup{i}{z}_\ast$ with
$(1-\,\lsup{i}{z}_\ast) \otimes y+ y \otimes\,\lsup{i-1}{z}_\ast,$
which, by induction, is identified with
$\frac{\partial(y^i \cdot  z_\ast \cdot  \overline y^i )}{\partial(y,  z_\ast  )}.$
It follows that, in the quotient of $\Z[G \times (x,y,\,\lsup{[0{\uparrow}\nu]}{z}_{[1{\uparrow}d]})]$,
$\frac{\partial (r(x,\lsup{[0{\uparrow}\nu]}{z}_{[1{\uparrow}d]} ))}
                            {\partial(x,y,\,\lsup{[0{\uparrow}\nu]}{z}_{[1{\uparrow}d]})} $
becomes identified with $\frac{\partial (r(x,y, z_{[1{\uparrow}d]} ))}
                            {\partial(x,y, z_{[1{\uparrow}d]})} $,
and the   exact quotient sequence is the augmented cellular chain
complex of the adjusted  universal cover of our original  Cayley complex.
\end{proof}

In the case where $G$ is a hyperbolic orientable \ors group,
this result was obtained by Howie~\cite[Corollary~3.6]{Howie04}.

\section{VFL and Euler characteristics}\label{sec:VFL}

In this section we calculate the Euler characteristics of the \ors groups.

\begin{Defs} Consider any   resolution of $\Z$ by projective,
left $\Z G$-mod\-ules
\begin{equation} \label{Eresolution}
 \cdots \longrightarrow P_2 \longrightarrow P_1
\longrightarrow P_0 \longrightarrow \Z \longrightarrow 0.
\end{equation}

The \textit{length} of the projective $\Z
G$-resolution~\eqref{Eresolution} is the
supremum, in
$[0{\uparrow}\infty]$, of the non-empty set $\{n \in
[0{\uparrow}\infty] \mid P_{n} \ne 0\}$.

Let $\mathcal{P}$ denote the unaugmented complex $\cdots \longrightarrow P_2 \longrightarrow P_1
\longrightarrow P_0  \longrightarrow 0$, and
view $\Q$  as a $\Q$-$\Z G$-bi\-module. For each $n \in [0{\uparrow}\infty[$,
$ \Hop_n(G; \Q) \coloneqq  \Tor_n^{\Z G}(\Q, \Z)  \coloneqq \Hop_n(\Q \otimes_{\Z G} \mathcal{P})$;
for the purposes of this article,    it will be convenient to understand that $\Hop_n(G; -)$
applies to \textit{right} $\Z G$-modules.
The \textit{$n$th  Betti number of $G$} is
$b_n(G) \coloneqq\dim_{\mathbb {Q}} \Hop_n (G;\mathbb {Q})  \in [0{\uparrow}\infty]$.
The value of the Betti numbers does not depend on the choice of
projective $\Z G$-resolu\-tion~\eqref{Eresolution}.

The \textit{cohomological dimension of $G$},
denoted $\cd G$, is the smallest element
of the  subset of $[0{\uparrow}\infty]$ which consists of  lengths of
 projective $\Z G$-resolu\-tions of $\Z$.
The \textit{virtual cohomological dimension of $G$},
denoted $\vcd G$, is the smallest element of the subset of $[0{\uparrow}\infty]$
which consists of cohomological dimensions of  finite-index subgroups of $G$.
The  \textit{cohomological dimension of $G$ with respect to an associative ring $Q$},
denoted $\cd_{{\scriptscriptstyle Q}}G$, is the smallest element
of the  subset of $[0{\uparrow}\infty]$ which consists of  lengths of
 projective $Q G$-resolu\-tions of $Q$.

If there exists a resolution~\eqref{Eresolution} of finite length
such that all the $P_n$ are finitely generated, free  left $\mathbb {Z}G$-modules,
then we say that $G$ is \textit{of type $\FL$}  and we define the \textit{Euler characteristic of}  $G$
to be $$\chi(G)\coloneqq\textstyle\sum\limits_{n \in [0{\uparrow}\infty]} (-1)^n b_n(G).$$
If $G$ has some finite-index  subgroup $H$ of type $\FL$, then we say that  $G$ is  \textit{of type VFL} and,
if $G$ is not of type FL,
we  define the   \textit{Euler characteristic of}   $G$ to be
$$\chi(G) \coloneqq \textstyle\frac{1}{[G:H]}\chi(H);$$
this value, which is sometimes called the `virtual
 Euler characteristic', does not depend on the choice of~$H$.
\hfill\qed
\end{Defs}

\begin{Rem}\label{Rem:vor} If $G$ has some finitely generated, one-relator, finite-index
subgroup $H$, say $H = \gp{X}{r}$, then
$G$ is of type VFL  and $$\textstyle -\chi(G) = \frac{-1}{[G:H]} \chi(H) =
\frac{1}{[G:H]} (\abs{X} - 1 - \frac{1}{\log_{\gp{X}{\,\,}}(r)}).$$
See, for example,~\cite[($*$) on
 p.~167]{Chiswell03}.  \hfill\qed
\end{Rem}

We now discuss the case where $r = x^m$.

\begin{Exs}\label{Exs:vor}  Let $d,\,m \in[0{\uparrow}\infty[$\,,
let $F \coloneqq \gp{(x,y)\vee z_{[1{\uparrow}d]}}{\quad}$,
and let $u$ and $r$ be elements of~$F$.  Suppose that
$u \in \gen{  z_{[1{\uparrow}d]}} - \{1\} $ and that $r = x^m$.
Let  $G = \gp{(x,y)\vee z_{[1{\uparrow}d]}}{[x,y]u, r}$.

In $[1{\uparrow}\infty]$, let $m_F\coloneqq \log_Fx^m \in \{m,\infty\}$, let
$m'$ denote the supremum of the orders of the finite subgroups of~$G$,
and let $m''\coloneqq \max(m_F ,m')$.

It is  straightforward to verify the following assertions.  It then follows that
$G$ is virtually one-relator,
and hence of type VFL.
\begin{enumerate} [\normalfont (i).]
\vskip-0.6cm \null
\item If $m=0$, then
$G\,\, = \,\, \gp{(x,y)\vee z_{[1{\uparrow}d]}}{[x,y]u}\vspace{1mm}$.

Here,
 $m_F= \infty$, $m' = 1$,  $m'' = \infty$, and
$-\chi(G) = d = d - \frac{1}{m''} \ge 0$.
\vskip-0.7cm \null
\item  If  $m=1$, then   $G\,\, = \,\,\gp{(y)\vee z_{[1{\uparrow}d]}}{u}$.

Here,  $m_F = 1$, $m' = \log_{F}u$,  $m'' = \log_{F}u$,  and
$-\chi(G) = d   - \frac{1}{\log_{F}u} =   d-\frac{1}{m''}\ge 0$.
\vskip-0.7cm \null
\item  If $m\ge 2$,  then  $G \,\, = \,\,\langle\, (y)\vee \lsup{\Z_m}{z}_{[1{\uparrow}d]}
 \mid \textstyle\prod\limits_{i\in [(m-1){\downarrow}0]} \lsup{i}{u}
 \,\rangle
\,\,\,\rtimes\,\,\,\gp{x}{x^m}$, with the $x$-ac\-tion defined by $\lsup{x}{y}  =  \lsup{0}{\overline u} \cdot y$ and
$\lsup{x}{(\,\lsup{i}{z}_\ast )} = \lsup{i+1}{z}_\ast$.

Here,
 $m_F = m' = m'' = m$, and $-\chi(G) =  d-\frac{1}{m}  =  d-\frac{1}{m''}\ge 0 $. \hfill\qed
\end{enumerate}
\end{Exs}

It is now convenient to
 go back to writing  $(x,y)\vee z_{[1{\uparrow}d]}$ in the form
$x_{[1{\uparrow}{k}]}$ with $k=d+2$.

\begin{Cor}\label{Cor:Qexact} Let $k \in [3{\uparrow}\infty[$\,, let $F\coloneqq  \gp{x_{[1{\uparrow}{k}]}}{\quad}$,
and let $w$ and $r$ be elements of~$F$.  Suppose that
 $w \in [x_1,x_2]\gen{x_{[3{\uparrow}k]}}$ and that $r$ is a Hempel relator for
$\prs{x_{[1{\uparrow}{k}]}}{w}$.  Let\linebreak   $G\coloneqq \gp{x_{[1{\uparrow}k]}}{w, r}\vspace{1mm}$.

Let $m \coloneqq \log_{F}r$.  Then  $\vcd G \le 2$, $G$ is of type VFL and $-\chi(G) = k-2 - \frac{1}{m} \ge 0\vspace{1mm}$.

Let $C$ denote the subgroup of $G$ generated by
the image of $\sqrt[F]{r}$, let $Q$ be any associative ring such that $mQ = Q$, let $R\coloneqq Q G$, and
let $e\coloneqq \frac{1}{m} \sum\limits_{c\in C}c \in R\vspace{-1.5mm}$.
Then the sequence of left $R$-modules\vspace{-1mm}
\begin{equation*}
0 \to R \oplus Re
\xrightarrow{
\begin{pmatrix}
 \frac{ \partial w}{\partial x_{1}} &\cdots & \frac{ \partial w}{\partial x_{k}}
 \\[.2cm] \frac{ \partial r}{\partial x_{1}}  &\cdots & \frac{ \partial r}{\partial x_{k}}
\end{pmatrix}
}
R^{k}
\xrightarrow{   \begin{pmatrix}
\scriptstyle x_1-1 \\[.1cm]
\scriptstyle\vdots\\[.1cm]
\scriptstyle x_k-1
\end{pmatrix} }
R \xrightarrow{G\to \{1\}} Q \to 0
\end{equation*}
is exact, and $\cd_Q G \le 2$.
\end{Cor}

\begin{proof}
By Theorem~\ref{thm:torsion}(\ref{it:10}), there exists some torsion-free finite-index subgroup $H$ of $G$.
Clearly, $H$ acts freely on $G$ on the left and the number of orbits is $[G{\colon}H] = \abs{H\leftmod G}$.
Since $H$ meets each conjugate of $C$ trivially, $C$ acts freely on $H\leftmod G$ on the right,
 $H$  acts freely on $G/C$\vspace{1mm} on the left,
and the number of orbits in each case is $\abs{H\leftmod G \rightmod C} =
 \frac{\abs{H\leftmod G}}{\abs{C}} = \frac{[G{:}H]}{m}$.  \vspace{1mm}

By Theorem~\ref{Thm:exact}, with $d = k-2$, we see that
$\cd H \le 2$,
$H$ is of type FL, and $\chi(H) =  [G{:}H] - (d+2)[G{:}H] + ([G{:}H] + \frac{[G{:}H]}{m})$.
Hence $\vcd G \le 2$,
 $G$ is of type VFL, and $\chi(G) = 1 - (d+2) + (1 + \frac{1}{m})$.
Since $(QC)e \simeq Q$ as left $QC$-modules, we see that $$Q[G/C] = Q G\otimes_{QC} Q \simeq
Q G\otimes_{QC} (QC)e     \simeq Re.$$
  It is now clear that the
results hold.
\end{proof}

We can now describe the Euler characteristics of the \ors groups.

\begin{Rems}\label{Rems:ors} Let $k \in [0{\uparrow}\infty[$\,,
let $F\coloneqq  \gp{x_{[1{\uparrow}{k}]}}{\quad}$,
and let $w$ and $r$ be elements of $F$.
Suppose
either that $k$ is even and
$w = \prod_{i \in [1{\uparrow}\frac{k}{2}]} [x_{2i-1},x_{2i}]\vspace{.5mm}$,
or that  $k \ge 1$ and $w = \prod x^2_{[1{\uparrow}k]}$.
Let $S \coloneqq \gp{x_{[1{\uparrow}k]}}{w}$ and let
  $G\coloneqq  \gp{x_{[1{\uparrow}k]}}{w,r}$.\vspace{1mm}

In $[1{\uparrow}\infty]$, let $m\coloneqq \log_S r\,\text{mod}\,w$, let
$m'$ denote the supremum of the orders of the finite subgroups of $G$,
and let $m''\coloneqq \max(m,m')$.

\medskip

\noindent\textbf{Case 1.} $k \le 2$.

Here, $G$ is virtually abelian of rank at most two.
In  particular, $G$ is  virtually  one-relator, and
it follows from Remark~\ref{Rem:vor} that $G$ is of type VFL
and $\chi(G) = \frac{1}{\abs{G}} \ge 0$.

\medskip

\noindent\textbf{Case 2.} $k  \ge 3$ and, for some $m \in [0{\uparrow}\infty[$\,,
$r\,\,\text{mod}\,w$ is conjugate in $S$ to the $m$th
power of some free generator of the subgroup
 $\gp{x_1x_2,\, x_2x_3}{\quad }$ of $S$.

Here, by using Lemma~\ref{Lem1}  together with   Examples~\ref{Exs:vor},
we find that $G$ is virtually  one-relator,
that $G$ is of type VFL, and that $-\chi(G) =  k-2 - \frac{1}{m''}  \ge 0$.
There are two possibilities for~$m''$.

\medskip

\noindent\textbf{Case 2a.}  $(k,m) = (3,1)$.

That is, $r\,\,\text{mod}\,w$ is conjugate in $S$ to some free generator of the
rank-two, free subgroup
 $\gp{x_1x_2,\, x_2x_3,\, x_3x_1 }{(x_1x_2)(x_2x_3)(x_3x_1)}$ of $S$.
Here, $G \simeq C_\infty{\ast}C_2$,  $-\chi(G)  = \frac{1}{2} \ge 0$, $m = 1$,
$m' = 2$, and $m'' = 2$.

\medskip

\noindent\textbf{Case 2b.} $(k,m) \ne (3,1)$.

Here,  $-\chi(G) = k-2 - \frac{1}{m} \vspace{1mm}\ge 0$.
If  $r\,\,\text{mod}\,w \ne 1$, then  $m=m'=m'' < \infty$.
If  $r\,\,\text{mod}\,w = 1$, then  $m= \infty$, $m' = 1$, and $m'' = \infty$.

\medskip

\noindent\textbf{Case 3.} $k  \ge 3$ and
$r\,\,\text{mod}\,w$ is not conjugate in $S$ to any
power of any free generator of the subgroup
 $\gp{x_1x_2,\, x_2x_3}{\quad }$ of $S$.

By using Lemmas~\ref{Lem1} and~\ref{Lem:Hempelments},
we find that there exists some presentation for $G$ as in Corollary~\ref{Cor:Qexact}.

Hence,  $G$ is of type VFL, $-\chi(G) =  k-2 - \frac{1}{m} \vspace{1mm}\ge 0$, and $m = m' = m'' < \infty$.
\hfill\qed
\end{Rems}

We  summarize the preceding, but do not record that
one can use $m'' = m$ except for one case, or $m'' = m'$ except for one case.

\begin{Cor}\label{Cor:chiors} Let $k \in [0{\uparrow}\infty[$\,,
let $F\coloneqq  \gp{x_{[1{\uparrow}{k}]}}{\quad}\vspace{1mm}$,
and let $w$ and $r$ be elements of $F$.
Suppose
either that $k$ is even and
$w = \prod_{i \in [1{\uparrow}\frac{k}{2}]} [x_{2i-1},x_{2i}]\vspace{1mm}$,
or that  $k \ge 1$ and $w = \prod x^2_{[1{\uparrow}k]}$.
Let $S \coloneqq \gp{x_{[1{\uparrow}k]}}{w}$ and let
  $G\coloneqq  \gp{x_{[1{\uparrow}k]}}{w,r}\vspace{1mm}$.
In $[1{\uparrow}\infty]$, let $m\coloneqq \log_S r\,\text{\normalfont mod}\,w$, let
$m'$ denote the supremum of the orders of the finite subgroups of $G$,
and let $m''\coloneqq \max(m,m')$. Then

$$\chi(G) = \begin{cases} \frac{1}{\abs{G}} \ge 0 &\text{if $k \le 2$,}\\[.3cm]
-k+2+\frac{1}{m''} \le 0 &\text{if $k \ge 3$}.
\end{cases}$$
\vskip -.8cm\hfill\qed\vskip .6cm
\end{Cor}

\section{$L^2$-Betti numbers of \ors groups}\label{Sors}\label{sec:L2}

Let us begin with a brief algebraic review of Atiyah's theory of $L^2$-Betti numbers of groups.

\begin{Rev}\label{Rev}
Let $\mathbb{C}[[G]]$ denote the set of all functions from $G$ to $\mathbb{C}$
expressed as formal sums, that is, a function $x \colon G \to \mathbb{C}$, $g \mapsto x_g$,
will be written as $\sum\limits_{g \in G} x_g g$. Then $\mathbb{C}[[G]]$
has a natural $\mathbb{C}G$-bimodule structure, and contains
a copy of $\mathbb{C}G$ as $\mathbb{C}G$-sub-bimodule\vspace{1mm}.
For each $x \in \mathbb{C}[[G]]$,  let $\norm{x}\coloneqq {{ \sqrt{\phantom{\sum_{g\in G}dd}}}
\hskip-1.6cm \sum\limits_{g\in G}\, \abs{x_g}^2}\in [0,\infty]$\vspace{1mm},
and let $\tr(x) \coloneqq x_{1_G} \in \mathbb{C}$.

Let
$\ell^2(G) \coloneqq \{x \in \mathbb{C}[[G]] :  \norm{x} < \infty\}.$
 For $x$, $y \in \ell^2(G)$, $g \in G$, and $S$ a finite subset of $G$,
it follows from the Cauchy-Schwarz inequality that
$\sum\limits_{h \in S} \abs{x_h y_{\,\overline h g}}  \le  \norm{x}{\cdot}\norm{y}$;
hence, there exists a well-defined limit
$\sum\limits_{h \in G} ( x_h  y_{\,\overline h g}) \in \C$;
hence, there exists a well-defined element
 $x{\cdot} y\coloneqq
\sum\limits_{g \in G} ((\sum\limits_{h\in G} x_hy_{\,\overline hg}) g) \in \C[[G]]$,
 called the \textit{external product} of $x$ and $y$.

The \textit{group von Neumann algebra of $G$} is defined as the additive abelian group
$$\mathcal{N}(G)\coloneqq
\{p \in \ell^2(G) \mid p {\cdot} \ell^2(G) \subseteq \ell^2(G)\}$$
endowed with the ring structure induced by the external product;
it can be shown that this agrees with the definition  of the group von Neumann algebra
given in~\cite[Section 1.1]{Lueck02}, and that $\mathcal{N}(G) =
\{p \in \ell^2(G) \mid  \ell^2(G){\cdot} p  \subseteq \ell^2(G)\}$.  We then have a chain of $\C G$-bimodules
 $  \mathbb{C}G \subseteq \mathcal{N}(G) \subseteq  \ell^2(G) \subseteq \mathbb{C}[[G]]$,
and $\mathbb{C}G$ is a subring of $\mathcal{N}(G)$, and
$\ell^2(G)$ is an $\mathcal{N}(G)$-bi\-mod\-ule containing $\mathcal{N}(G)$ as
$\mathcal{N}(G)$-sub-bi\-mod\-ule.

It can be shown that the elements of $\mathcal{N}(G)$ which act faithfully on
the left $\mathcal{N}(G)$-mod\-ule~$\ell^2(G)$  are precisely
the two-sided non-zerodivisors in $\mathcal{N}(G)$,
and that these form a left and right Ore subset of
$\mathcal{N}(G)$; see~\cite[Theorem~8.22(1)]{Lueck02}.
The \textit{ring of unbounded operators affiliated to $\mathcal{N}(G)$},
denoted  $\mathcal{U}(G)$, is defined as
the left, and the right,
Ore localization of $\mathcal{N}(G)$ at the set of its two-sided non-zerodivisors;
see~\cite[Section 8.1]{Lueck02}.
For example, it is then clear that
\begin{equation}\label{Einvertible}
\text{if $g$ is an element of $G$ of infinite order, then $g-1$ is invertible in $\mathcal{U}(G)$.}
\end{equation}

It can be shown that $\mathcal{U}(G)$ is a von Neumann regular ring in which one-sided
inverses are two-sided inverses, and, hence,
one-sided zerodivisors are two-sided zerodivisors;
see~\cite[Section 8.2]{Lueck02}.

It can be shown that there exists  a continuous, additive von Neumann dimension
that assigns to every left $\mathcal{U}(G)$-module $M$
a value
$\dim_{\mspace{2mu}\mathcal{U}(G)}M \in [0,\infty]$;
see Definition 8.28 and Theorem 8.29 of~\cite{Lueck02}.
For example,  if $e$ is an idempotent element of $\mathcal{N}(G)$, then
$\dim_{\mspace{2mu}\mathcal{U}(G)}\mathcal{U}(G)  e= \tr(e);$
see Theorem 8.29 and Sections~6.1-2 of~\cite{Lueck02}.

For each $n \in [0{\uparrow}\infty[$\,,
the \textit{$n$th $L^2$-Betti number of $G$} is defined as
$$b_n^{(2)}(G) \coloneqq \dim_{\mspace{2mu}\mathcal{U}(G)}\,\,
\Hop_n(G; \mathcal{U}(G))\,\,\,\,\in\,\,\,\, [0,\infty],$$
where  $\mathcal{U}(G)$ is to be
viewed as a $\mathcal{U}(G)$-$\Z G$-bi\-module; see
 Definition~6.50, Lemma~6.51 and Theorem 8.29 of~\cite{Lueck02}.

It is easy to show that if $G$ is finite, then, for each $n \in [0{\uparrow}\infty[$\,,
$$b_n^{(2)}(G) = \begin{cases}
 \chi(G) = \frac{1}{\abs{G}} &\text{if $n = 0$,}\\
0 &\text{if $n \in [1{\uparrow}\infty[\,$.}
\end{cases}$$

By~\cite[Theorem~6.54(8b)]{Lueck02},
\begin{equation}\label{eq:zero}
b_0^{(2)}(G) = \textstyle\frac{1}{\abs{G}}.
\end{equation}

By~\cite[Theorem 1.9(8)]{Lueck02}, if $H$ is a finite-index subgroup
of $G$, then
\begin{equation}\label{eq:index}
 b_n^{(2)}(H) = [G:H] b_n^{(2)}(G).
\end{equation}

In general, there is little relation between the $n$th
$L^2$-Betti number
and the $n$th (ordinary)  Betti number,
$b_n(G) = \dim_{\mathbb {Q}} \Hop_n (G;\mathbb {Q}) \in [0{\uparrow}\infty].$
However, by~\cite[Remark~6.81]{Lueck02}, if $G$ is of type VFL, we can define and calculate
the \textit{$L^2$-Euler characteristic}
\begin{equation}\label{eq:chi}
\chi^{(2)}(G)\coloneqq\textstyle\sum\limits_{n \in [0{\uparrow}\infty[} (-1)^n b_n^{(2)}(G) =
\sum\limits_{n \in [0{\uparrow}\infty[} (-1)^n b_n (G)\eqqcolon\chi(G).
\end{equation}
\vskip -1.1cm\hfill\qed \vskip .8cm
\end{Rev}

Recall that, for any finitely generated, virtually one-relator group $G$, a formula for $\chi(G)$ was given in
Remark~\ref{Rem:vor}.

 \begin{Lem}\label{Lem:vor} If $G$ is a finitely generated, virtually
one-relator  group, then
$G$ is of type VFL and, for each $n \in [0{\uparrow}\infty[$\,,
\begin{equation*}
b_n^{(2)}(G) = \begin{cases}
\max\{\chi(G),0\} = \frac{1}{\abs{G}} &\text{if $n = 0$,}\\
\max\{-\chi(G),0\} &\text{if $n = 1$,}\\
0 &\text{if $n \in [2{\uparrow}\infty[\,$.}
\end{cases}
\end{equation*}
\end{Lem}

\begin{proof} If $G$ is itself a  finitely generated, one-relator group,
the conclusions
hold, by~\cite[Theorem~4.2]{DicksLinnell}.  By Review~\ref{Rev}\eqref{eq:index}, this then
extends to overgroups of finite
index on dividing by the index.
\end{proof}

For Hempel presentations we have the following information.

\begin{Lem}\label{Lem:text} Let $k \in [3{\uparrow}\infty[$\,,
let $F\coloneqq  \gp{x_{[1{\uparrow}{k}]}}{\quad}\vspace{1mm}$, and
let $w$ and $r$ be elements of~$F$.  Suppose that
$w \in [x_1,x_2]\gen{x_{[3{\uparrow}k]}}\vspace{1mm}$ and that $r$
is a Hempel relator for $\prs{x_{[1{\uparrow}{k}]}}{w}$. Let\linebreak
$G\coloneqq \gp{x_{[1{\uparrow}k]}}{w, r}\vspace{1mm}$.
 Let $m \coloneqq \log_{F}r$,  let $C_m$ denote the subgroup of $G$ generated by
the image of $\sqrt[F]{r}$, and
let $e\coloneqq \frac{1}{m} \sum\limits_{c\in C_m}c \in \C G$.
\begin{enumerate}[\normalfont (i).]
\vskip-0.4cm \null
\item For all $q \in \mathcal{U}(G) $  and all $a \in  \mathbb{C}G$,
 if $qea = 0$ then either $qe= 0$ or $ea = 0$.
\vskip-0.7cm \null
\item The  homology of\vspace{-3mm}\newline\null\hskip 1cm
$
0 \to \mathcal{U}(G) \oplus \mathcal{U}(G)\,e
\xrightarrow{
\begin{pmatrix}
 \frac{ \partial w}{\partial x_{1}} &\cdots & \frac{ \partial w}{\partial x_{k}}
 \\[.2cm] \frac{ \partial r}{\partial x_{1}}  &\cdots & \frac{ \partial r}{\partial x_{k}}
\end{pmatrix}
}
\mathcal{U}(G)^{k}
\xrightarrow{   \begin{pmatrix}
\scriptstyle x_1-1 \\[.1cm]
\scriptstyle\vdots\\[.1cm]
\scriptstyle x_k-1
\end{pmatrix} }
\mathcal{U}(G)  \to 0,\vspace{5mm}
$
\newline is $\Hop_*(G; \mathcal{U}(G))$, and, for each
 $n \in   (0) \vee [3{\uparrow}\infty[$\,,\, $b_n^{(2)}(G) = 0$.
\end{enumerate}
\end{Lem}

\begin{proof} (i). Recall that~\cite[Theorem~3.1(iii)]{DicksLinnell}, which depends on
results in~\cite{BurnsHale72} and~\cite{Linnell92},
asserts that if $G$  is (( free  $\rtimes~C_m$) by (locally indicable)),
then (i) holds.  By Theorem~\ref{thm:torsion}(ii),(iii),
we now see that (i) holds.

(ii). By using
Corollary~\ref{Cor:Qexact} and Review~\ref{Rev}\eqref{eq:zero}, it is not difficult to
see that (ii) holds.
\end{proof}

Let us now consider the special case of hyperbolic  \ors groups.

\begin{Lem}\label{Lem:main} Let $k \in [3{\uparrow}\infty[$\,,
let $F\coloneqq  \gp{x_{[1{\uparrow}{k}]}}{\quad}\vspace{1mm}$,
and let $w$ and $r$ be elements of $F$.  Suppose
either that  $w= [x_1,x_2]\prod x_{[3{\uparrow}k]}^2$ or that $k$ is even and
$w = \prod_{j\in[1{\uparrow}\frac{k}{2}]}[x_{2j-1},x_{2j}]$.
Suppose that
 $r$  is a Hempel relator for $\prs{x_{[1{\uparrow}{k}]}}{w}$.
Let $G\coloneqq \gp{x_{[1{\uparrow}k]}}{w, r}\vspace{1mm}$.
 Then    $b_2^{(2)}(G) = 0$.
\end{Lem}

\begin{proof}  Let $m \coloneqq \log_{F}r$,
let $C$ denote the subgroup of $G$ generated by
the image of $\sqrt[F]{r}$, and
let $e\coloneqq \frac{1}{m} \sum\limits_{c\in C}c \in \C G$.
If the map
$$
\mathcal{U}(G) \oplus \mathcal{U}(G)\,e
\xrightarrow{
\begin{pmatrix}
 \frac{ \partial w}{\partial x_{1}} &\cdots & \frac{ \partial w}{\partial x_{k}}
 \\[.2cm] \frac{ \partial r}{\partial x_{1}}  &\cdots & \frac{ \partial r}{\partial x_{k}}
\end{pmatrix}
}
\mathcal{U}(G)^{k}
$$
is injective, then $b^{(2)}(G) = 0$, by~Lemma~\ref{Lem:text}(ii).
It remains to consider the case where there exists
some $(p,qe) \ne (0,0)$ in the kernel of this map,
that is, $p$,\,$ q \in \mathcal{U}(G)$ and
\begin{equation}\label{Eker1}
\text{ for each $j \in [1{\uparrow}k]$,} \quad  p \textstyle\frac{ \partial w}{\partial x_j}
 +  qe\frac{ \partial r}{\partial x_j}  = 0 \text { in }
\mathcal{U}(G).
\end{equation}

Let $G$ act
by right multiplication
on the set of all
subsets of $\mathcal{U}(G)$,  let $V \coloneqq p + qe\mathbb{C}G$, a subset of $\mathcal{U}(G)$,
and let $G_V \coloneqq  \{g \in G \mid Vg =  V\}= \{g \in G \mid p(g-1) \in qe\C G\}$, a subgroup of~$G$.
The subset
$V$ of $\mathcal{U}(G)$ is then closed under the right $G_V$-action on $\mathcal{U}(G)$.
By Lemma~\ref{Lem:text}(i), the surjective map
$e\mathbb{C}G \to qe\mathbb{C}G$, $ea \mapsto qea$,
is either injective or zero.  In either event,
$qe\mathbb{C}G$ is a projective right
$\mathbb{C}G$-module, and hence a projective right $\mathbb{C}G_V$-module.
By the left-right dual
of~\cite[Corollary~5.6]{DicksDunwoody06},
there exists a right $G_V$-tree $T$ with finite edge stabilizers
and vertex set the right $G_V$-set   $p+ qe\C G$.

We claim that $0 \not \in p+ qe\C G$.  Since $(p,qe)\ne (0,0)$, we may assume that $qe \ne0$ for
this argument.
Consider any $a\in \mathbb{C}G$.
By Corollary~\ref{Cor:Qexact}\vspace{2mm}, with $Q = \C$,
the map $$\mathbb{C}G \oplus \mathbb{C}Ge
\xrightarrow{
\begin{pmatrix}
 \frac{ \partial w}{\partial x_{1}} &\cdots & \frac{ \partial w}{\partial x_{k}}
 \\[.2cm] \frac{ \partial r}{\partial x_{1}}  &\cdots & \frac{ \partial r}{\partial x_{k}}
\end{pmatrix}
}
\mathbb{C}G^{k}$$ is injective, and, in particular, $(-ea,e)$ does not lie in the kernel, since $e \ne 0$.
Hence there exists some $j \in [1{\uparrow}k]$,
such that $\textstyle 0 \ne -ea \frac{ \partial w}{\partial x_j}
 + e \frac{ \partial r}{\partial x_j}
= e(-a \frac{ \partial w}{\partial x_j}  +   \frac{ \partial r}{\partial x_j})\vspace{1mm}$.
By Lemma~\ref{Lem:text}(i),
$$\textstyle 0 \ne qe(-a \frac{ \partial w}{\partial x_j}  +   \frac{ \partial r}{\partial x_j})
=  (-qea) \frac{ \partial w}{\partial x_j}  +  qe \frac{ \partial r}{\partial x_j}.$$
It is now clear from~\eqref{Eker1} that $p \ne -qea$.  This proves that $0 \not\in p + qe\C G$, as claimed.

By using Review~\ref{Rev}\eqref{Einvertible}, we now find that each vertex stabilizer for $T$ is torsion.
Hence, by  Theorem~\ref{thm:torsion}\eqref{it:12},  each vertex stabilizer for $T$
has order at most $m$.
It follows that $G_V$ is virtually free; see, for
example,~\cite[Theorem~IV.1.6,(b)$\Rightarrow$(c)]{DicksDunwoody89}.

Let $$G_{\pm V} \coloneqq \{g \in G \mid\text{ either } p(g-1)\in qe\mathbb{C}G \text{ or } p(g+1) \in qe\mathbb{C}G \}
= \{g \in G \mid Vg = \pm V \},$$ a
subgroup of $G$.   We shall see that  $G_{\pm V} = G$.

By~\eqref{Eker1}, we see that,  for each $j \in [1{\uparrow}k]$,
 $p\,\frac{\partial w}{\partial x_j} \in qe\mathbb{C}G$, and, hence,
\begin{align}
 p\,(1 - x_1x_2\overline x_1 ) = p\,\textstyle
\frac{\partial w}{\partial x_1} &\in qe\mathbb{C}G,\label{Eker21}
 \\
  p\,(x_1 - [x_1,x_2]) = p\,\textstyle
\frac{\partial w}{\partial x_2}  &\in qe\mathbb{C}G,\label{Eker22}
\\
\text{for all $j \in [3{\uparrow}k]$, \,\, }p\, \textstyle
\frac{\partial w}{\partial x_j} &\in qe\mathbb{C}G. \label{Eker24}
\end{align}
By~\eqref{Eker21}, $\lsup{x_1}{x_2} \in G_V$.
Right multiplying~\eqref{Eker22} by $\overline x_1$, we see that
$ p\,(1-\lsup{x_1x_2}{\overline x_1}) \in qe\mathbb{C}G$.  Thus
$\lsup{x_1x_2}{x_2}$ and $\lsup{x_1x_2}{x_1}$ lie in $G_{V}$.  Hence, their product $x_1x_2$ lies in $G_{V}$,
and then $x_1$ and $x_2$ lie in~$G_{V}$.
In particular, $V[x_1,x_2]=V$.

We claim that $x_{[3{\uparrow}k]} \subseteq G_{\pm V}$.
Arguing inductively, we suppose that $j \in [3{\uparrow}k]$
and   $x_{[1{\uparrow}(j-1)]} \subseteq G_{\pm V}$. \vspace{1mm}
For this step, we shall consider only the non-orientable case, $w=[x_1,x_2]\prod x_{[3{\uparrow}k]}^2$; the
\vspace{1mm} orientable case is similar, and was done in~\cite[Lemma~5.15]{DicksLinnell}.
Let $u = [x_1,x_2]\textstyle\prod
 x_{[3{\uparrow}(j-1)]}^2$.  Then $u\in G_V$, that is, $p-pu \in qe\C G$.
Now $\frac{\partial w}{\partial x_j} = u(1+x_j)$, and,
by~\eqref{Eker24},   $pu(1+x_j) \in qe\C G$. Summing these two elements of $qe\C G$,
 we see that $p + pux_j \in qe\C G$.
Thus $ux_j \in G_{\pm V}$, and, hence, $x_j \in G_{\pm V}$.

By induction, the claim is proved.

Hence $G_{\pm V} = G$, and,  hence $[G{:}G_V]\le 2$, and, hence, $G$ is virtually free.
  Thus $\vcd G \le 1$, and, hence, $b_2^{(2)}(G) = 0$.
\end{proof}

Recall that, for any \ors group $G$, a formula for $\chi(G)$ was given in
Corollary~\ref{Cor:chiors}.

\begin{Thm} Let $G$ be a \ors group.  Then $G$ is of type VFL
and, for each $n \in [0{\uparrow}\infty[$\,,
\begin{equation*}
b_n^{(2)}(G) = \begin{cases}
\max\{\chi(G),0\} = \frac{1}{\abs{G}} &\text{if $n = 0$,}\\
\max\{-\chi(G),0\} &\text{if $n = 1$,}\\
0 &\text{if $n \in [2{\uparrow}\infty[\,$.}
\end{cases}
\end{equation*}
\end{Thm}

\begin{proof} If $G$ is   virtually one-relator, then, by Lemma~\ref{Lem:vor}, the desired conclusions hold.
Thus, we may assume that $G$ is \textit{not}  virtually  one-relator.
It then follows from Remarks~\ref{Rems:ors} that
we may assume that $G$ has a presentation as in Lemma~\ref{Lem:main} and
that $G$ is of type VFL.  Then,  by Lemma~\ref{Lem:main} and
Lemma~\ref{Lem:text}(ii),
$b_2^{(n)}(G) = 0$ for all $n \in (0) \vee [2{\uparrow}\infty[$\,.
 By Review~\ref{Rev}\eqref{eq:chi},   $\chi(G) = -b_1^{(2)}(G)$, as desired.
\end{proof}

\newpage

\setcounter{section}{0}
\renewcommand\thesection{\text{A.}\arabic{section}}
\setcounter{Thm}{1}

\bigskip

\bigskip

\centerline{\LARGE\textbf{APPENDIX}}

\bigskip

\centerline{\LARGE\textbf{Howie towers via Bass-Serre Theory}}

\bigskip

In this appendix, we use Bass-Serre Theory to prove some of
Howie's results on local indicability.

We shall use~\cite{DicksDunwoody89} as our reference  for Bass-Serre theory.

Throughout, let $F$ be a group.

\section{Actions on trees}\label{BS}

\begin{Not}
Let  $E$ be a subset of an $F$-set $X$.

Two elements of $E$ that lie in the same $F$-orbit in $X$ are said to be
\textit{glued together by} $F$, and we write
 $\glue(F,E) \coloneqq \{g\in F \mid gE \cap E \ne \emptyset\}$.
\hfill\qed
\end{Not}

\begin{Defs} Let $T = (T,VT,ET, \iota,\tau)$ be an $F$-tree.

\medskip

\noindent(i). Let $r$ be an element of $F$   that   fixes no vertex of $T$.

The smallest $\gen{r}$-subtree of~$T$,
 denoted $\axis(r)$, has the form of the real line, and
$r$ acts  on it by shifting it; see,
for example,~\cite[Proposition~I.4.11]{DicksDunwoody89}.
We write $\Eaxis(r)\coloneqq E(\axis(r))$.

The \textit{$(F,ET)$-support} of $r$ is defined as
$$\supp(r)\coloneqq    \{Fe  \in F\leftmod ET \mid e \in \Eaxis(r)\}.$$

For all $f \in F$, $\axis(\lsup{f}{r}) = f\axis(r)$ and
$\supp(\lsup{f}{r}) = \supp(r)$.

For all $n \in [1{\uparrow}\infty[$\,,  $\axis(r^n) = \axis(r)$.

If  $F$ acts freely on $ET$, then $r$ has a root in $F$.

If $\glue(F, \Eaxis(r)) = \gen{r}$, then $\sqrt[F]{r} = r$.

\medskip

\noindent(ii). Let $<$ be a (total) ordering of $F\leftmod ET$.

A subset $R$  of $F$  is said to be \textit{$(F,ET,{<})$-staggered}
 if each element of $R$ fixes no vertex of~$T$ and, for each $(r_1, r_2) \in R \times R$,
exactly one of the following three conditions holds:
\vskip-0.8cm \null
\begin{list} {}
\item $\lsup{F}\!{r}_1 = \lsup{F}\!{r}_2$;
\item $\min (\supp(r_1), <) \,\,  < \,\, \min(\supp(r_2),<)$  and  $\max (\supp(r_1), <) \,\, < \,\, \max(\supp(r_2),<)$;
\vskip-0.7cm \null
\item $\min (\supp(r_2), <) \,\,  < \,\, \min(\supp(r_1),<)$   and  $\max (\supp(r_2), <) \,\, < \,\, \max(\supp(r_1),<)$.
\end{list}

\medskip

\noindent When either of the latter two conditions holds,
we say that $r_1$ and $r_2$ have \textit{staggered supports} (with respect to $<$).

\medskip

\noindent(iii). A subset $R$  of $F$  is said to be \textit{$(F,ET)$-staggerable}
 if
there exists an ordering $<$  of $F\leftmod ET$
such that $R$ is $(F,ET,{<})$-staggered.
\hfill\qed
\end{Defs}

\begin{Ex}\label{Ex:free}
For a free product $F = A{\ast}B$,
the \textit{Bass-Serre tree} is the $F$-graph  $T$ with  vertex set   the disjoint union of
$F/A$ and $F/B$
and  edge set   $F$, in which  each $f \in ET=F$ has initial vertex $fA$
and terminal vertex $fB$.
It can be shown that $T$ is a tree; see, for example,~\cite[Theorem~I.7.6]{DicksDunwoody89}.
Notice that $F$ acts \textit{freely} on the edge set of $T$.

Let $r \in F$.

Clearly, $r$  fixes some vertex of $T$ if and only if $r$ lies in some conjugate of
$A$ or some conjugate of $B$.

If   $r$ fixes no vertex of $T$, then
there exists some $n \in [1{\uparrow}\infty[$\,, and
some sequence $a_{[1{\uparrow}n]} $ in $ A -\{1\}$, and some sequence $b_{[1{\uparrow}n]}$ in $ B-\{1\}\vspace{1mm}$
such that some conjugate of $r$
can be expressed in the form
$\lsup{f}{r} =    \textstyle \prod_{i\in[1{\uparrow}n]} (a_ib_i)\vspace{1mm}.$
The entire conjugacy class $\lsup{F}{r}$  can be represented by writing
$\prod_{i\in[1{\uparrow}n]} (a_ib_i)$ cyclically. \vspace{1mm}

Here, $\supp(r) = \{F \} = F\leftmod ET$.  If $<$ denotes the unique ordering of $F\leftmod ET$, then
$\{r\}$ is $(F,ET,{<})\vspace{1mm}$-staggered.

If $r=    \textstyle \prod_{i\in[1{\uparrow}n]} (a_ib_i)\vspace{1mm}$,
then $\gen{r}\leftmod \axis(r) = \{\gen{r}w \mid w \text{ is an initial subword of } r \}$
and there is  a direct description of $\sqrt[F]{r}$ and $\log_Fr$, as follows.
We can write $\log_Fr = \frac{n}{m}$, where $m$ is smallest element of
  $ [1{\uparrow}n]$ with the property  that, for each $i \in [1{\uparrow}(n-m)]$,
$a_i = a_{i+m}$ and $b_i = b_{i+m}$.  Here,
$\sqrt[F]{r} = \prod\limits_{i\in[1{\uparrow}m]} (a_ib_i)$.  \hfill\qed
\end{Ex}

\section{Staggerability}

\begin{Defs}  A group is said to be \textit{indicable} if \textit{either} it  is  trivial \textit{or} it   has
some infinite, cyclic quotient.
A group is said to be \textit{locally indicable} if every finitely generated subgroup is indicable.
\hfill\qed
\end{Defs}

For any subset $R$ of $F$ in which each element of $R$ has a root in $F$,
we let $\sqrt [F]R$ denote the set of roots of elements of $R$ in $F$.
In this section, we let $F$ be a locally indicable group,
and let $T$ be an $F$-tree with trivial edge stabilizers, and
let $R$  be an $(F,ET)$-staggerable subset of $F$ with $\sqrt[F]{R} = R$.
We will prove that  $F/\gen{\lsup{F}\!{R}}$ is locally indicable,
and   that $\gen{\lsup{F}\!{R}}$ acts freely on $T$,
or, equivalently,  the natural map
$F \onto F/\gen{\lsup{F}\!{R}}$ is injective on the vertex stabilizers.
The following result deals with an extreme case.

\begin{Lem}\label{Lem:ind} Let $F$ be a locally indicable group,
let $T$ be an $F$-tree with trivial edge stabilizers, and let
$R$ be an $(F,ET)$-staggerable subset of $F$ such that $\sqrt[F]{R}=R$.

Suppose that $F$ is finitely generated and that $F\leftmod T$ is finite.\vspace{1mm}

Then $F/\gen{\lsup{F}\!{R}}$ is indicable; if $F/\gen{\lsup{F}\!{R}}$ is trivial then
$F$ acts freely on $T$ and, for each $r \in R$,
$\glue(F, \Eaxis(r)) =  \gen{r}$.
\end{Lem}

\begin{proof}
We argue by induction on $\abs{F\leftmod ET}$.
Since  $F$ is a finitely generated, locally indicable group,
$F$  is indicable, which means that the implications hold when $R$ is empty.
Thus we  may assume that $R$ is  non-empty;
in particular, $\abs{F\leftmod ET} \ge 1$.
By induction, we may suppose  that   the
implications hold for all smaller values of $\abs{F\leftmod ET}$.

We may assume that $R = \lsup{F}\!{R}$, and we may  assume that we are given
an ordering   $<$ of the finite set $F\leftmod ET$
such that $R$ is $(F,ET, {<})$-stag\-gered.
In particular,  there exists some
$e_\maxx \in \bigcup\limits_{r  \in R} \Eaxis(r)$
such that\vspace{-5mm}  $$Fe_\maxx
\quad = \quad \max(\{Fe \mid e \in \bigcup\limits_{r  \in R} \Eaxis(r)\},\,\, <).\vspace{-2mm}$$
There then exists some $r_\maxx \in R$ such that $e_\maxx \in \Eaxis(r_\maxx)$,
and, by the definition of $(F,ET,{<})$-stag\-gered,
$Fe_\maxx$ does not meet the axis of any element of
$R - \lsup{F}\!{(r_\maxx)}$.
Thus there exists some pair $(r,e)$, for example, $(r_\maxx, e_\maxx)$, such that the following hold.
\begin{align}
&r \in R \text{, } e \in \Eaxis(r) \text{, and }
  Fe \text{ does not meet the axis of any element of } R- \lsup{F}\!{r}. \label{eq:disj}
 \\
&\glue(F, \Eaxis(r)) = \gen{r}
\text{ and/or }  (r,e) = (r_\maxx,e_\maxx).\label{eq:max}
\end{align}

In the forest $T - Fe$,  let $T_\iota$ denote the component containing
$\iota e$, and let $T_\tau$ denote the component containing $\tau e$.
Let $F_\iota$ denote the $F$-stabilizer of $\{T_\iota\}$,
and let   $F_\tau$ denote the $F$-stabilizer of $\{T_\tau\}$.
Let $R_\iota \coloneqq R \cap F_\iota$ and $R_\tau \coloneqq R \cap F_\tau$.
By~\eqref{eq:disj}, for each $r' \in R - \{\lsup{F}\!{r}\}$, $\axis(r')$ lies in $T - Fe$ and hence lies in a
component of $T - Fe$.  It follows that $\lsup{F}\!{R}_\iota  \,\cup \,\lsup{F}\!{R}_\tau \, \cup\,  \lsup{F}\!{r}$ is all of $R$.
Notice that if $F\{T_\iota\} = F\{T_\tau\}$ then  $\lsup{F}\!{R}_\iota  = \lsup{F}\!{R}_\tau$.

By applying the Bass-Serre Structure Theorem to
 the $F$-tree whose vertices are the components of $T - Fe$,
and whose edge set is $Fe$, with $fe$ joining $f T_\iota$ to $f T_\tau$, we see that
\begin{equation}\label{eq:preconn}
F = \begin{cases}
F_\iota {\ast} F_\tau &\text{if $F\{T_\iota\} \ne F\{T_\tau\}$,}\\
F_\iota {\ast} \gp{f}{\quad}&\text{if  $f \in F$ and $f T_\iota  =   T_\tau $.}
\end{cases}
\end{equation}
Hence,
\begin{equation}\label{eq:conn}
F/\normgen{R-\{r\}} =
\begin{cases}
(F_\iota/\normgen{R_\iota}) {\ast} (F_\tau/\normgen{R_\tau}) &\text{if $F\{T_\iota\} \ne F\{T_\tau\}$,}\\
(F_{\iota}/\normgen{R_{\iota}}) {\ast} \gp{f}{\quad}&\text{if  $f \in F$ and $f T_\iota  =   T_\tau $.}
\end{cases}
\end{equation}

Consider the case where both
$F_\iota/\normgen{R_\iota}$ and $F_\tau/\normgen{R_\tau}$ have
infinite, cyclic quotients.  By~\eqref{eq:conn},
$F/\normgen{R-\{r\}}$ has a rank-two, free-abelian quotient.
On incorporating $r$, we see that $F/\normgen{R}$ has an
infinite, cyclic quotient, and the desired conclusion holds.

Thus, it remains to consider the case where one of
$F_\iota/\normgen{R_\iota}$,  $F_\tau/\normgen{R_\tau}$  does not have an
infinite, cyclic quotient; by replacing $e$ with $\overline e$, if necessary,
we may assume  that $F_\iota/\normgen{R_\iota}$ does not have an infinite,
cyclic quotient.

By the induction hypothesis applied to $(F_\iota, T_\iota, R_\iota)$,
we see that $F_{\iota}/\normgen{R_{\iota}}$ is trivial,
that $F_{\iota}$ acts freely on $T_{\iota}$, and,
for each $r_{\iota} \in R_{\iota}$, $\glue(F_\iota, \Eaxis(r_\iota)) =  \gen{r_\iota}$,
and, hence, $\glue(F, \Eaxis(r_\iota)) =  \gen{r_\iota}$.

By replacing $r$ with $\overline r$ if necessary, we may assume the following.
\begin{equation}\label{eq:path}
\text{There exists a  segment of $\axis(r)$ of the form $e, p, re$.}
\end{equation}

If $F\{T_\iota\} \ne F\{T_\tau\}$, then, by~\eqref{eq:path}, we have a path
$\overline r p$ in $\axis(r)$ from $\overline r \tau e\!\in\!\overline r T_\tau\!\ne\!T_\iota$\linebreak to $\iota e \in T_\iota$.
Now $\overline r p$ necessarily enters $T_\iota$ through an edge of the form $g\overline e $ where\linebreak
$g \in F_\iota$ and $\overline e$ is the inverse of the edge $e$.
Notice that $g \in \glue(F_\iota, \Eaxis(r)) -\gen{r}$. This proves the following.
\begin{equation}\label{eq:disconnected2}
\text{If   $\glue(F_\iota, \Eaxis(r)) \subseteq \gen{r}$  then
$F\{T_\iota\} = F\{T_\tau\}$.}
\end{equation}

Consider the case where $R_\iota$ is empty.
Here, $F_\iota\!=\!F_{\iota}/\normgen{R_{\iota}}\!=\!\{1\}$.
By~\eqref{eq:disconnected2},   $F\{T_\iota\} = F\{T_\tau\}$,\linebreak and, then,
by~\eqref{eq:preconn},  $F = \gp{t}{\quad}$.
Now $\{t, \overline t\} = \sqrt[F]{F} \supseteq R = \{r\}$.
Hence, $F = \gen{r}$, and, hence, $\glue(F, \Eaxis(r)) = \gen{r}$
 This proves the following.
\begin{equation}\label{eq:empty}
\text{If  $\glue(F, \Eaxis(r)) \ne \gen{r}$ then $R_\iota$ is non-empty.}
\end{equation}

\medskip

\noindent\textbf{Case 1.}   $\glue(F, \Eaxis(r)) = \gen{r}$.

Here, by~\eqref{eq:disconnected2},   $F\{T_\iota\} = F\{T_\tau\}$.
Hence, $F(VT_\iota) = VT$ and $R = \lsup{F}\!{R}_\iota  \cup  \lsup{F}\!{r}$.
Let $v \in VT$.  We wish to show that $F_v = 1$,
and, we may assume that $v \in VT_\iota$.
Here $F_v \le F_{\iota}$, and, since $F_\iota$ acts freely on $T_\iota$,  $F_v = 1$, as desired.
Thus $F$ acts freely on $T$.
In~\eqref{eq:path}, the path $p$ from $\tau e$ to $r\iota e$ in $\axis(r)$ does not
meet $Fe$, and, hence, $p$ stays within $T_\tau$, and, hence $r\iota e \in T_\tau$.
Thus, $r T_\iota = T_\tau$, and, by~\eqref{eq:preconn},   $F = F_\iota {\ast} \gp{r}{\quad}$.
Since  $F_{\iota}/\normgen{R_{\iota}}$ is trivial, we see that
$F/\normgen{R}$ is trivial, and, hence, $F/\normgen{R}$ is indicable.
Here all the required conclusions hold.

\medskip

\noindent \textbf{Case 2.} $\glue(F, \Eaxis(r)) \ne \gen{r}$.

By~\eqref{eq:empty},  $R_\iota$ is non-empty, and, hence,  there exists some
$e_\iota \in \bigcup\limits_{r_\iota  \in R_\iota} \Eaxis(r_\iota)$
such that\vspace{-2mm}  $$Fe_\iota
= \min(\{Fe \mid e \in \bigcup\limits_{r_\iota  \in R_\iota} \Eaxis(r_\iota)\},\,\, <).\vspace{-1mm}$$
 There then exists some
 $r_\iota \in R$ such that $e_\iota \in \Eaxis(r_\iota)$, and we
then know that\linebreak $\glue(F, \Eaxis(r_\iota))\!=\!\gen{r_\iota}$.
By~\eqref{eq:max}, $(r,e)\!=\!(r_\maxx,e_\maxx)$.
Using the definition of $(F,T,{<})$-stag\-gered, one can show that
$Fe_\iota$ does not meet the axis of any element of
$R - \lsup{F}\!{r}_\iota$.
We then replace $(r,e)$ with $(r_\iota, e_\iota)$,
and, by Case 1, all the required conclusions hold.

\medskip

This completes the proof.
\end{proof}

We next deal with the case where $R$ is finite.

\begin{Thm}\label{Thm:howie0} Let $F$ be a locally indicable group,
let $T$ be an $F$-tree with trivial edge stabilizers, and let
$R$ be an $(F,ET)$-staggerable subset of $F$ such that $\sqrt[F]{R}= R$.

Suppose that  $R$ is finite and  that $H$ is a finitely generated subgroup
of $F$  such that $H$ contains $R$, and
 $H/\gen{\lsup{H}\!{R}}$ has no infinite, cyclic quotient.

Then there exists some finitely generated subgroup $F'$ of $F$
such that  $F'$ contains $H$, $R$ is $(F',ET)$-staggerable,
$F'/\gen{\lsup{F'}\!{R}}$ is trivial,
$F'$ acts freely on $T$ and, for each $r \in R$,
$\glue(F', \Eaxis(r)) =  \gen{r}$.
Here, $H \le F' = \gen{\lsup{F'}\!{R}} \le  \gen{\lsup{F}\!{R}}$, and, hence,  $H$
acts freely on $T$.
\end{Thm}

\begin{proof}  Let $v$ be an arbitrary vertex of $T$.
Choose a finite generating set $S$ of $H$ such that $S$ contains $R \cup \{1\}$, and let
$Y$ be the smallest subtree of $T$ containing $Sv$.
Then, for each $s \in S$, $sVY \cap VY $ is non-empty since it contains $sv$; thus
\begin{equation}\label{eq:Y}
S \subseteq \glue(F,VY).
\end{equation}
If $F'$ is any subgroup of $F$ and $T'$ is any
$F'$-subtree of $T$, for the purposes of this proof
let us say that  $(F',T')$ is an \textit{admissible} pair
if  $F' \supseteq S$, and   $T' \supseteq Y$, and
$R$ is $(F',ET')$-staggerable.
By hypothesis, $(F,T)$ is admissible.
Let $<$ be an ordering of $F\leftmod ET$ such that $R$ is
$(F,ET,{<})$-staggered.

\medskip

\noindent \textbf{The Type 1 transformation.}  Suppose that $F \ne \gen{S \cup \,\, \glue(F, EY)}$ or $T \ne FY$.
Define $F'\coloneqq\gen{S \cup \,\, \glue(F, EY)}$ and $T'\coloneqq F'Y$.  We shall prove that
$(F',T')$  is an admissible pair
and $\glue(F,EY) =\glue(F',EY)$.

\medskip

Clearly
\begin{equation}\label{eq:y}
\glue(F,EY) =\glue(F',EY).
\end{equation}
It follows from~\eqref{eq:Y} that $$S \,\,\cup\,\, \glue(F', EY) \quad \subseteq \quad \glue(F',Y).$$
If we consider the set of components of the $F'$-forest $F'Y$ in $T$ as an $F'$-set,
we see that the component containing $Y$ is fixed by a generating set of $F'$, and
hence is fixed by $F'$.  This component must then be all of $F'Y$.
Hence $T'$ is connected,  and, hence, $T'$ is an $F'$-tree.

It is straightforward to show that
\begin{equation}\label{eq:t'}
\glue(F, ET')   =  F'.
\end{equation}
By~\eqref{eq:t'},  the natural map
from $F'\leftmod ET'$ to $F \leftmod ET$ is an embedding;
we again denote by $<$ the ordering of $F'\leftmod ET'$
induced from $F\leftmod ET$.
We claim that, with the conjugation action by $F$ on $F$,
$\glue(F,R) \subseteq F'$. Suppose that $f \in F$, that $r \in R$, and that   $\lsup{f}{r} \in R$.
Then $\axis(r)$ and $\axis(\lsup{f}{r})$ lie in $T'$
and are glued together by $f$ since $ \axis(r^f) = f\axis(r)$.
By~\eqref{eq:t'},   $f \in F'$, and the claim is proved.
It follows that $R$ is $(F',ET',{<})$-staggered.
Hence $(F',T')$ is admissible.
This completes the verification of the Type 1 transformation.

\medskip

\noindent \textbf{The Type 2 transformation.}  Suppose that $F = \gen{S \cup \,\, \glue(F, EY)}$,
that $T = FY$, and that  $F/\gen{\lsup{F}\!{R}}$  has some infinite, cyclic quotient $F/N$;
here, $N$ is a normal subgroup of $F$ such that $N \geq \gen{\lsup{F}\!{R}}$.
We shall prove that $(N,T)$ is admissible
and $\glue(N, EY) \subset \glue(F, EY)$.

\medskip

Since $H/\gen{\lsup{H}\!{R}}$ has no infinite, cyclic quotient, it follows that $H \subseteq N$.

Since $F/N$ is cyclic, we can choose $x \in F$ such that $xN$ generates $F/N$.
Then $F = \gen{x}N$.  Since $F/N$ is infinite,
 $\gen{x} \cap N = \{1\}$.

We now  give $N\leftmod ET$ an ordering.
Give $(F\leftmod ET) \times \Z$ the lexicographic ordering and
choose a left $F$-transversal $E_0$ in $ET$.
Then $\gen{x}E_0$ is a left
$N$-transversal in $ET$, and there is a bijective map
$$(F\leftmod ET) \times \Z \quad \to \quad  N\leftmod ET, \quad (Fe, n) \mapsto  Nx^ne,  (e,n)  \in E_0 \times \Z.$$
Let $<_N$ denote the ordering induced on $N\leftmod ET$ by this bijection.

We claim that $R$ is $(N,ET,{<_N})$-staggered.
To see this, consider any  $r_1$, $r_2 \in R$ such that $\lsup{N}{r}_1 \ne \lsup{N}{r}_2$.
If $\lsup{F}{r}_1  = \lsup{F}{r}_2$, then $r_1$ and $r_2$ will have equal
$(F,ET)$-supports, while their $(N,ET)$-supports, viewed in $(F\leftmod ET) \times \Z$,
 will differ by a non-zero shift in the second coordinate,
and hence be staggered.
If $\lsup{F}{r}_1  \ne \lsup{F}{r}_2$, then   $r_1$ and $r_2$ will have  staggered
$(F,ET)$-supports,  and, hence, their $(N,ET)$-supports, viewed in
$(F\leftmod ET) \times \Z$, will also be staggered.

Hence $(N,T)$ is admissible.

Since   $N \not \supseteq F = \gen{S \cup \glue(F,EY)}$, we see that
$N \not \supseteq  S \cup \glue(F,EY)$.
Since $N \supseteq S$,  we see that
$N \not \supseteq  \glue(F,EY)$, and, hence, $\glue(N,EY) \subset \glue(F,EY)$.
This completes the verification of the Type 2 transformation.

\medskip

Since $EY$ is finite and edge stabilizers are trivial, $\glue(F, EY)$ is
finite. Notice that transformations of Type 2 reduce   $\glue(F, EY)$,
while transformations of Type 1 do  not change  $\glue(F, EY)$.
Notice that we cannot apply two  transformations of  Type 1 consecutively.
  Thus, after applying a finite number of  transformations of
Types 1 and  2, we arrive at a  pair $(F',T')$
such that $F' = \gen{S \cup \glue(F',EY)}$, $T' = F'Y$, $R$ is $(F',ET')$-staggerable, and
$F'/\gen{R^{F'}}$  has no infinite, cyclic quotient.
By Lemma~\ref{Lem:ind},  $F'/\gen{\lsup{F'}\!{R}}$ is trivial,
$F'$ acts freely on $T'$, and, for each $r \in R$,
$\glue(F', \Eaxis(r)) =  \gen{r}$.   Since $F'$ acts freely on $T'$,
it follows that $F'$ acts freely on all of $T$.  Since $R$ is $(F',ET')$-staggerable,
it follows that $R$ is $(F',ET)$-staggerable, because any
ordering on $F'\leftmod ET'$ can be extended to some ordering of
$F'\leftmod ET$, by the axiom of choice.
\end{proof}

The finite descending chain of subgroups implicit in the
above argument is the chain of subgroups considered by Howie in
his tower arguments.

We now have a general result.

\begin{Cor}\label{Cor:main} Let $F$ be a locally indicable group,
let $T$ be an $F$-tree with trivial edge stabilizers, and let
$R$ be an $(F,ET)$-staggerable subset of $F$.

Then $\gen{\lsup{F}\!{R}}$ acts freely on $T$,
that is, each vertex stabilizer embeds in $F/\gen{\lsup{F}\!{R}}$ under the natural map.

If, moreover,  $\sqrt[F]{R}= R$, then $F/\gen{\lsup{F}\!{R}}$ is locally indicable.
\end{Cor}

\begin{proof} Since $\gen{\lsup{F}\!{R}} \le \gen{\lsup{F}\!{(\sqrt[F]{R})}}$,
and $\sqrt[F]{R}$ is again $(F,ET)$-staggerable,
we may assume that $\sqrt[F]{R}= R$.

 We first show that
 $\gen{\lsup{F}\!{R}}$ acts freely on $T$.
Let $R'$ be an arbitrary finite subset of $\lsup{F}\!{R}$,
and let $H = \gen{R'}$.  On applying
Theorem~\ref{Thm:howie0}, we see that
$\gen{R'}$ acts freely on $T$.  It then follows that
all of  $\gen{\lsup{F}\!{R}}$ acts freely on $T$.

We now show that $F/\gen{\lsup{F}\!{R}}$ is locally indicable.
Consider an arbitrary
finitely generated subgroup of $F/\gen{\lsup{F}\!{R}}$,
and express it in the form
$(\gen{S}\,\gen{\lsup{F}\!{R}})/\gen{\lsup{F}\!{R}}$ where
 $S$ is a finite subset of~$F$.
It remains to show that $(\gen{S}\,\gen{\lsup{F}\!{R}})/\gen{\lsup{F}\!{R}}$ is indicable.
We may assume that $(\gen{S}\,\gen{\lsup{F}\!{R}})/\gen{\lsup{F}\!{R}}$
has no infinite, cyclic quotient, and it remains to show that
$\gen{S} \le \gen{\lsup{F}\!{R}}$.

For the purposes of this proof, let  $\gen{S}'$ denote the derived subgroup of $\gen{S}$,
and let $\gen{S}^{\text{ab}}$ denote the abelianization $\gen{S}/\gen{S}'$.
Now $$(\gen{S}\,\gen{\lsup{F}\!{R}})/(\gen{S}'\gen{\lsup{F}\!{R}})
 = ((\gen{S}\,\gen{\lsup{F}\!{R}})/\gen{\lsup{F}\!{R}})^{\text{ab}}$$ which is a finite abelian group by supposition.
Let $d$ denotes its exponent.  Then\linebreak $S^d  \subseteq \gen{S}'\,\gen{\lsup{F}\!{R}}$
 and, hence, there exists some finite subset
 $R_0$ of $\lsup{F}\!{R}$ such that  $S^d$ lies in the set $\gen{S}'\gen{R_0}$.
Then, $(\gen{S \cup R_0}/\normgen{R_0})^{\text{ab}}$ is an abelian group of exponent at most~$d$,
and, hence, $\gen{S \cup R_0}/\normgen{R_0}$ has no infinite, cyclic quotient, and, hence,
by Theorem~\ref{Thm:howie0}, $\gen{S \cup R_0} \le  \gen{\lsup{F}\!{R}}$, as desired.
\end{proof}

\section{Consequences}

The main application is the following.

\begin{Thm}[Howie]\label{Thm:Howie} Let $A$ and $B$ be locally indicable groups and let
 $r$ be an element of $ A{\ast} B$ such that
$r$ is not conjugate to any element of $(A \cup B)-\{1\}$.  Then
the natural maps from $A$ and $B$ to $(A {\ast} B)/\normgen{r}$ are injective.
If, moreover, $\sqrt[A {\ast} B]{r} = r$,   then
$(A {\ast} B)/\normgen{r}$  is locally indicable.
\end{Thm}

\begin{proof}  Let  $T$ be the Bass-Serre  $(A {\ast} B)$-tree as in Example~\ref{Ex:free};
then $T$ is an $(A {\ast} B)$-tree with trivial edge stabilizers.

If $H$ is some finitely generated subgroup of $A {\ast} B$,
then the Bass-Serre Structure Theorem for the $H$-action on $T$,
or the Kurosh Subgroup Theorem, shows that $H$ is a free product of a family of
finitely generated groups each of which is free, or
isomorphic to a subgroup of $A$, or isomorphic to a subgroup of $B$.
Hence all these free factors are indicable, and hence $H$ is indicable.
Thus $A {\ast} B$ is locally indicable.  Thus we may assume that $r \ne 1$;
then $r$ is not conjugate to any element of $A \cup B$
and  $\{r\}$ is $(A {\ast} B,ET)$-staggerable.

By Corollary~\ref{Cor:main}, the natural maps from $A$ and $B$
to $(A {\ast} B)/\normgen{r}$ are injective, and, if $\sqrt[A{\ast} B]{r} = r$,
 then $(A {\ast} B)/\normgen{r}$ is locally indicable.
\end{proof}

\begin{Cor}[Magnus' Freiheitssatz]\label{Cor:Magnus} Let $F_1$ and $F_2$ be free groups.
If $r$ is an element of $F_1{\ast} F_2$ such that
$r$ is not conjugate to any element of $F_1-\{1\}$, then
the natural map from $F_1$  to $(F_1 {\ast} F_2)/\normgen{r}$ is injective. \hfill\qed
\end{Cor}

\begin{Cor}[Brodski\u{\i}]\label{Cor:Brodskii} Let $F$ be a free group.
If $r \in F$ and $ \sqrt[F]{r}=r$, then
$F/\normgen{r}$  is locally indicable. \hfill\qed
\end{Cor}

\begin{Rems}\label{Rems:Howie}
The injectivity result in
Theorem~\ref{Thm:Howie} is called the
\textit{local  indicability Freiheitssatz} since it generalizes Magnus' Freiheitssatz, Corollary~\ref{Cor:Magnus}.

The local indicability Freiheitssatz was proved independently by Brodski\u{\i}~\cite[Theorem~1]{Brodskii84},
Howie~\cite[Theorem~4.3]{Howie81}, and Short~\cite{Short84}.  The proof by Brodski\u{\i} was algebraic,
while the proofs by Howie and Short were topological, with Howie using topological towers and Short using diagrams.
B.~Baumslag~\cite{BBaumslag84} rediscovered Brodski\u{\i}'s algebraic proof.

The local indicability conclusion in Theorem~\ref{Thm:Howie} was proved by
Howie~\cite[Theorem~4.2(iii)$\Rightarrow$(i)]{Howie82},
by topological-tower methods.
The one-relator case, Corollary~\ref{Cor:Brodskii},
had been proved earlier by Brodski\u{\i}~\cite[Theorems~1 and~2]{Brodskii84}.
Howie~\cite{Howie00} later gave a direct proof of Brodski\u{\i}'s result
using elementary groupoid methods, and his
groupoid proof of the special case
led us to the Bass-Serre proof of the general case.
\hfill\qed
\end{Rems}

\begin{Cor}\label{Cor:Howie2} Let $G$ be a locally indicable group, let $F$ be a free group,
and let $r $ be an element of $G{\ast}F$ that is not a proper power and
that is not conjugate to any element of $G$.
Then $(G{\ast}F)/\normgen{r}$ is locally indicable.
\end{Cor}

\begin{proof}  Consider first the case where $r$ is conjugate to  an element of $F$.
By conjugating~$r$, we may  assume that $r \in F$.  Here,
$F/\normgen{r}$ is locally indicable by Corollary~\ref{Cor:Brodskii}.
Now $(G {\ast} F)/\normgen{r} = G {\ast} (F/\normgen{r})$  is locally indicable by the degenerate
 case of Theorem~\ref{Thm:Howie}.
Thus we may assume that $r$ is not conjugate to any element of $(G \cup F) -\{1\}$.
Here, $(G {\ast} F)/\normgen{r}$ is locally indicable by Theorem~\ref{Thm:Howie}.
\end{proof}

It is now straightforward to deduce the following
 special case of a    result of Howie~\cite[Corollary~4.5]{Howie82} on `reducible presentations'.

\begin{Cor}[Howie]\label{Cor:redpres}  Let $G_{[0{\uparrow}\infty[}$ be a family of groups
such that $G_0=1$ and, for all $n \in [0{\uparrow}\infty[$,
 $G_{n+1}= (G_{n} {\ast} \gp{X_{n+1}}{\quad})/\normgen{r_{n+1}}$ where $X_{n+1}$ is a set,
and $r_{n+1} $ is an element of $ G_{n} {\ast} \gp{X_{n+1}}{\quad}$  that
 is not a proper power and  that is not conjugate to any element of~$G_n$.
Then $\bigcup G_{[0{\uparrow}\infty[}$ is locally indicable.
\hfill\qed
\end{Cor}

\bigskip

\noindent{\textbf{\Large{Acknowledgments}}}

\medskip
\footnotesize

The research of the first- and second-named authors was
jointly funded by the MEC (Spain) and the EFRD~(EU) through Project  MTM2006-13544.

\bibliographystyle{amsplain}

\textsc{Departament de  Matem\`atiques,
Universitat Aut\`onoma de Barcelona,
E-08193 Bella\-terra (Barcelona), Spain}

\emph{E-mail address}{:\;\;}\texttt{yagoap@mat.uab.cat}

\medskip

\textsc{Departament de  Matem\`atiques,
Universitat Aut\`onoma de Barcelona,
E-08193 Bella\-terra (Barcelona), Spain}

\emph{E-mail address}{:\;\;}\texttt{dicks@mat.uab.cat}

\emph{URL}{:\;\;}\texttt{http://mat.uab.cat/$\scriptstyle\sim$dicks/}

\medskip

\textsc{Department of Mathematics,
Virginia Tech,
Blacksburg, VA 24061-0123, USA}

\emph{E-mail address}{:\;\;}\texttt{linnell@math.vt.edu}

\emph{URL}{:\;\;}\texttt{http://www.math.vt.edu/people/linnell/}
\end{document}